\documentclass[11pt,a4paper]{article}

\newcommand{\thetitle}{Estimating time-changes in noisy L\'evy
  models}%
\newcommand{\forenames}{Adam D.}
\newcommand{\surname}{Bull}
\newcommand{\fullname}{\forenames\ \surname}
\newcommand{\acks}{We thank EPSRC for their support under grant EP/K000993/1}

\newcommand{\mscone}{91G70}

\newcommand{\msctwo}{62G08}
\newcommand{\mscthree}{62G20}
\newcommand{\mscfour}{62G35}
\newcommand{\themsc}{\mscone\ (primary); \msctwo, \mscthree,
  \mscfour\ (secondary)}

\newcommand{\kwdone}{It\=o semimartingale}
\newcommand{\kwdtwo}{L\'evy process}
\newcommand{\kwdthree}{microstructure noise}
\newcommand{\kwdfour}{volatility}
\newcommand{\kwdfive}{time-change}
\newcommand{\thekeywords}{\kwdone, \kwdtwo, \kwdthree, \kwdfour,
  \kwdfive}

\newcommand{\addressone}{Statistical Laboratory}%
\newcommand{\addresstwo}{University of Cambridge}%

\usepackage[round]{natbib}
\usepackage{amsmath, amsfonts, amssymb, amsthm, mathtools, 
  booktabs, tikz, mparhack, longtable, tabu, booktabs, needspace}
\usepackage[bookmarks, breaklinks, colorlinks, linkcolor=blue,
citecolor=blue]{hyperref}
\hypersetup{pdftitle={\thetitle}, pdfauthor={\fullname}}

\let\oldmarginpar\marginpar
\renewcommand\marginpar[1]{\-\oldmarginpar[\raggedleft\footnotesize 
#1]{\raggedright\footnotesize #1}}

\DeclarePairedDelimiter{\abs}{\lvert}{\rvert}
\DeclarePairedDelimiter{\norm}{\lVert}{\rVert}
\newcommand{\N}{\mathbb{N}}
\newcommand{\Z}{\mathbb{Z}}
\newcommand{\C}{\mathbb{C}}

\newcommand{\R}{\mathbb{R}}
\renewcommand{\P}{\mathbb{P}}
\newcommand{\E}{\mathbb{E}}
\newcommand{\Var}{\mathbb{V}\mathrm{ar}}

\newtheorem{theorem}{Theorem}
\newtheorem{lemma}{Lemma}
\newtheorem{corollary}{Corollary}
\newtheorem{definition}{Definition}

\newcommand{\kl}{K_l}
\newcommand{\oh}{\tfrac12}
\newcommand{\clu}{c_{l/n_2}(u)} 
\newcommand{\hclu}{\widehat c_l(u)}
\newcommand{\hcmu}{\widehat c_m(u)}

\newcommand{\hcolu}{\widehat c_{1,l}(u)}
\newcommand{\hcolus}{\widehat c_{1,l}^2(u)}
\newcommand{\hctlu}{\widehat c_{2,l}(u)}
\newcommand{\hctlus}{\widehat c_{2,l}^2(u)}
\newcommand{\hcklu}{\widehat c_{k,l}(u)}
\newcommand{\vlu}{\varphi_{l/n_2}(u)}
\newcommand{\vku}{\varphi_{k/n_0}(u)}
\newcommand{\vlub}{\overline \varphi_l(u)}
\newcommand{\hvlu}{\widehat \varphi_l(u)}
\newcommand{\plu}{\psi_{l/n_2}(u)}
\newcommand{\pku}{\psi_{k/n_0}(u)}
\newcommand{\pksu}{\psi_{k/n_0}^2(u)}
\newcommand{\hplu}{\widehat \psi_l(u)}
\newcommand{\ssk}{\sigma^2_{k/n_0}}
\newcommand{\hssk}{\widehat \sigma^2_k}
\newcommand{\tssk}{\widetilde \sigma^2_k}
\newcommand{\hus}{\oh u^2}

\newcommand{\hxk}{\widehat X_k}
\newcommand{\hxkm}{\widehat X_{k,m}}
\newcommand{\ej}{\varepsilon_j}

\newcommand{\ejs}{\varepsilon_{j+1}}

\newcommand{\dxj}{\Delta X_j}
\newcommand{\dxm}{\Delta X_m}

\newcommand{\aj}{p_j}
\newcommand{\am}{p_m}
\newcommand{\bj}{q_j}
\newcommand{\bm}{q_m}
\newcommand{\ft}{\mathcal F_t}
\newcommand{\ftp}{\mathcal F_t^+}

\newcommand{\fz}{\mathcal F_0}

\newcommand{\fj}{\mathcal F_{j/n}}
\newcommand{\fjs}{\mathcal F_{(j+1)/n}}
\newcommand{\fm}{\mathcal F_{m/n}}
\newcommand{\fmp}{\mathcal F_{m/n}^+}
\newcommand{\fms}{\mathcal F_{(m+1)/n}}

\newcommand{\fjp}{\mathcal F_{j/n}^+}
\newcommand{\fk}{\mathcal F_{k/n_0}}

\newcommand{\fl}{\mathcal F_{l/n_2}}
\newcommand{\fls}{\mathcal F_{(l+1)/n_2}}

\newcommand{\tjik}{_{j \in J_k}}

\newcommand{\ejp}{\varepsilon_{j+1}}

\newcommand{\nj}{Z_{\varepsilon,j}}
\newcommand{\xij}{Z_{I,j}}

\newcommand{\el}{E_l}
\newcommand{\elc}{E_l^c}

\newcommand{\iam}{\mathcal I^\alpha(D,S)}
\newcommand{\iom}{\mathcal I^\alpha(D,1)}
\newcommand{\ihm}{\mathcal I^{1/2}(D,S)}
\newcommand{\wt}{\wedge S}
\newcommand{\emm}{\varepsilon_m}

\newcommand{\tjmik}{_{j,j+1\in J_k}}
\newcommand{\vtu}{\varphi_t(u)}
\newcommand{\ptu}{\psi_t(u)}
\newcommand{\ttu}{\rho^2_t(u)}
\newcommand{\tku}{\rho^2_{k/n_0}(u)}
\newcommand{\tlu}{\rho^2_{l/n_2}(u)}
\newcommand{\utu}{\tau^2_t(u)}
\newcommand{\hulu}{\widehat \tau^2_l(u)}

\newcommand{\ulu}{\tau^2_{l/n_2}(u)}
\newcommand{\ctu}{c_t(u)}
\newcommand{\tctu}{\widetilde c_t(u)}

\newcommand{\tcotu}{\widetilde c_{1,t}(u)}
\newcommand{\tcotus}{\widetilde c_{1,t}^2(u)}
\newcommand{\tcttu}{\widetilde c_{2,t}(u)}
\newcommand{\tcttus}{\widetilde c_{2,t}^2(u)}
\newcommand{\tchtu}{\widetilde c_{3,t}(u)}
\newcommand{\tchtus}{\widetilde c_{3,t}^2(u)}
\newcommand{\tcktu}{\widetilde c_{k,t}(u)}
\newcommand{\trlu}{\widetilde r_{l/n_2}(u)}
\newcommand{\hrlu}{\widehat r_{l}(u)}
\newcommand{\tcu}{\widetilde c(u)}
\newcommand{\trtu}{\widetilde r_t(u)}
\newcommand{\tctsu}{\widetilde c_t^S(u)}
\newcommand{\tru}{\widetilde r(u)}

\newcommand{\sa}{\mathcal S^{\alpha,\beta}}
\newcommand{\sh}{\mathcal S^{1/2,\beta}}
\newcommand{\sat}{\mathcal S^{\alpha,2}}
\newcommand{\ms}{\mathcal S}
\newcommand{\mt}{\mathcal T}
\newcommand{\sacm}{\sa(C,D)}
\newcommand{\satcm}{\sat(C,D)}
\newcommand{\sazcm}{\sa_{0}(C,D)}

\newcommand{\saecm}{\sa_{\gamma}(C,D)}
\newcommand{\ta}{\mathcal T^\alpha}
\newcommand{\tacm}{\ta(C,D)}
\newcommand{\toh}{\mathcal T^{1/2}}
\newcommand{\te}{\Omega_0}
\newcommand{\ten}{\Omega_{0,n}}
\newcommand{\ct}{c_t}
\renewcommand{\Re}{\mathrm{Re}}
\newcommand{\applink}{\autoref{sec:pro}}
\newcommand{\notlink}{}
\newcommand{\mainheader}{}
\newcommand{\suppheader}{}

\title{\thetitle\footnotetext{\emph{Acknowledgements:} \acks.}\footnotetext{\emph{MSC 2010:} \themsc.}
}
\author{\fullname\\\footnotesize \addressone\\\footnotesize \addresstwo}
\date{}

\begin{document}

\maketitle

\begin{abstract}
In quantitative finance, we often model asset prices as a noisy It\=o
semimartingale. As this model is not identifiable, approximating by a
time-changed L\'evy process can be useful for generative modelling. We
give a new estimate of the normalised volatility or time change in
this model, which obtains minimax convergence rates, and is unaffected
by infinite-variation jumps. In the semimartingale model, our estimate
remains accurate for the normalised volatility, obtaining convergence
rates as good as any previously implied in the literature.

\end{abstract}

\section{Introduction}

In quantitative finance, we often wish to predict the distribution of
future asset prices using historical data; this problem is of interest
when pricing options or evaluating investment strategies. From
economic considerations, we know that log-prices must be given by a
noisy semimartingale; however, this model cannot in general be
identified from price data.

We will therefore consider modelling log-prices as a noisy
time-changed L\'evy process. We note that this model is general enough
to describe the salient features of price data -- stochastic
volatility, jumps and noise -- while still being simple enough to
identify its parameters from data. It thus serves as a useful
approximation to the semimartingale model for generative modelling.

Our goal will be to estimate the normalised volatility or time-change
process in this model. Previous estimates have failed to achieve
minimax convergence rates when the jumps are of infinite variation, as
is suggested by empirical evidence.  We will therefore describe a new
estimate, which obtains minimax rates, and is unaffected by arbitrary
jump activity.

We will further show that in the semimartingale model, our estimate
remains accurate for the normalised volatility, obtaining convergence
rates as good as any previously implied in the literature. Our
estimate thus achieves the best of both worlds: good convergence when
the time-changed approximation is accurate, and no penalty when it is
not.


We begin by describing the statistical models we will consider. We
will suppose we have a single asset whose efficient log-price is given
by an {\em It\=o semimartingale},
\begin{equation}
\label{eq:sm}
X_t = X_0 + \int_0^t b_s \,ds + \int_0^t \sqrt{c_s}\, dB_s +
\int_{0}^t \int_\R x\,(\mu(dx,ds)
-1_{\abs{x}< 1}\,\nu_s(dx)\,ds),
\end{equation}
where \(b_t \in \R\) is a drift process, \(c_t > 0\) a volatility
process, \(\nu_t\) a jump measure process, \(\mu(dx,dt)\) a Poisson
random measure with intensity \(\nu_t(dx)\,dt,\) and the above
decomposition holds with respect to a filtration \(\ft.\) (We refer to
\citealp{jacod_limit_2003}, for definitions.)

We note that the assumption \eqref{eq:sm} is extremely common in
quantitative finance, and is motivated by economic no-arbitrage
arguments, as in \citet{delbaen_general_1994}. The model \eqref{eq:sm}
reproduces common features of price data, such as {\em stochastic
volatility} -- given by the dependence of the characteristics \((b_t,
c_t, \nu_t)\) on time -- and {\em jumps} -- given by the presence of the
jump measure process \(\nu_t.\)

To fit this model to price data, however, it is widely accepted that
we must also account for a third feature, known as {\em microstructure
  noise}. The quoted price of assets in general can diverge from the
efficient market price, due to economic artefacts such as the bid-ask
spread, tick sizes, transaction costs and communication delays.
Indeed, empirical studies confirm that high-frequency price data is
too volatile to be explained solely by an efficient price process
\citep{andersen_great_2000,mykland_comment:_2005,hansen_realized_2006}.

A popular model for microstructure noise is to assume that the
log-prices are observed under zero-mean errors. We thus consider
observations
\begin{equation}
\label{eq:ms}
Y_j = X_{Tj/n} + \ej, \quad j = 0, \dots, n-1,
\end{equation}
over a time interval \([0, T],\) and with errors \(\ej\) satisfying
\(\E[\ej \mid \mathcal F_{Tj/n}] = 0.\) (We refer to
\citealp{jacod_microstructure_2009}, for a discussion of this model.)

Unfortunately, the observations \(Y_j\) are insufficient to identify
the parameters of the model \eqref{eq:sm}. Even given noiseless
observations, letting the time horizon \(T \to \infty,\) and the step
size \(T/n \to 0,\) we cannot in general identify the drift process
\(b_t,\) or jump measure process \(\nu_t.\)


In the following, we will therefore also consider a {\em time-changed
  L\'evy process} model. Here, we instead suppose the log-price
\begin{equation}
\label{eq:tc}
X_t = L_{R_t},
\end{equation}
for a L\'evy process
\[
L_t = L_0 + b t + \sqrt{c} B_t + \int_{0}^t \int_\R x\,(\mu(dx,ds)
-1_{\abs{x}< 1}\,\nu(dx)\,ds),\] with drift \(b \in \R,\) volatility
\(c > 0,\) jump measure \(\nu,\) and Poisson random measure
\(\mu(dx,dt)\) with intensity \(\nu(dx)\,dt,\) and a time-change
process
\[R_t = \int_0^t r_s\,ds,\] given by a rate process \(r_t > 0.\)

The model \eqref{eq:tc} was popularised by
\citet{carr_time-changed_2004}, and its applications also discussed by
\citet{cont_financial_2004}.  Intuitively, this model describes prices
which move faster or slower according to an activity rate \(r_t;\)
this rate can be thought of as the cumulative effect of factors such
as trading activity and volume, investor liquidity, and general
economic uncertainty.

Formally, the time-changed model \eqref{eq:tc} is the subset of the
semimartingale model \eqref{eq:sm} which satisfies the separability
condition
\begin{equation}
\label{eq:tcsm}
b_t = b r_t,\qquad c_t = c r_t,\qquad\nu_t = \nu r_t.
\end{equation}
This condition requires, for example, that the jump measure \(\nu_t\)
be governed by the rate process \(r_t,\) and contain no idiosyncratic
jump component.

Since these parameters are defined only up to a multiplicative
constant, we must also choose a normalisation for \(r_t.\) In the
following, for simplicity we will set \(r_t\) to integrate to one
(although we will also discuss alternative
normalisations). Equivalently, using \eqref{eq:tcsm} we may define
\begin{equation}
\label{eq:def-r}
r_t = \frac{c_t}{\int_0^tc_s\,ds};
\end{equation}
we note that this definition is then also meaningful for the
semimartingale model~\eqref{eq:sm}.

The separability condition~\eqref{eq:tcsm} can be thought of as
similar to the additivity condition in an additive model. We take a
fully nonparametric model, which is difficult to fit, and restrict it
to a lower-dimensional one, which is less so. As our smaller
model~\eqref{eq:tc} reproduces the salient features of price data --
stochastic volatility, jumps and noise -- it can potentially offer a
good approximation to the full model~\eqref{eq:sm}.

This approximation can be useful in a variety of settings. If we wish
to predict the distribution of future asset prices, for example to
price options or evaluate investment strategies, we must fit a
generative model to the data. We already know we cannot fit the full
model~\eqref{eq:sm}, as we cannot identify its parameters \(b_t\) and
\(\nu_t.\) As we will see below, the parameters of the
model~\eqref{eq:tc} can all be identified from price data; it may thus
be used either directly as a generative model, or as a starting point
to identify suitable parametric alternatives
\citep{carr_time-changed_2004,cont_financial_2004}.

To fit the model~\eqref{eq:tc} to data, we must estimate the
parameters \(b,\,c,\,\nu,\) and \(r_t.\) If the time horizon \(T \to
\infty,\) and the step size \(T/n \to 0,\) the drift \(b\) and
volatility \(c\) can be estimated using standard
techniques. Estimation of the L\'evy measure \(\nu,\) while more
involved, has also been considered by several authors
\citep{figueroa-lopez_nonparametric_2009,figueroa-lopez_central_2011,belomestny_statistical_2011,belomestny_estimation_2012},
and extensions of \citeauthor{figueroa-lopez_nonparametric_2009}'s
approach to include noise are possible as in
\citet{vetter_inference_2013}.


In the following, we will focus specifically on estimation of the rate
process \(r_t.\) We first note that some of the factors contributing
to this process, in particular trading activity and volume, can be
observed directly.  While such side information may be useful in
practice, we can expect that not all such factors are observable, and
the relationship between observable factors and efficient prices may
be unclear, especially after accounting for microstructure noise.

In the following, we will therefore restrict ourselves to estimating
\(r_t\) directly from price data.  While previous work has provided
such estimates in a variety of settings
\citep{winkel_recovery_2001,woerner_inference_2007,rosenbaum_asymptotic_2010,figueroa-lopez_statistical_2012},
these authors have not considered our setting~\eqref{eq:ms}
and~\eqref{eq:tc}.  Even accounting for microstructure noise, we
cannot apply their methods here to obtain minimax rates of
convergence.

An alternative route to estimating \(r_t\) is to first use the
identification~\eqref{eq:tcsm}, and then estimate the volatility
\(c_t\) in the semimartingale model~\eqref{eq:sm}. Many authors have
described approaches for this problem, under various assumptions on
the jump measure process \(\nu_t.\)

If there are no jumps present, the integrated
volatility \(\int_0^1 c_s\,ds\) can be recovered using multiscale
estimators \citep{zhang_tale_2005,zhang_efficient_2006}, realised
kernels \citep{barndorff-nielsen_designing_2008}, or pre-averaging
\citep{jacod_microstructure_2009,podolskij_estimation_2009}.  The spot
volatility \(c_t\) can likewise be recovered using kernel estimators
\citep{kristensen_nonparametric_2010,mancini_spot_2012}, Fourier
series \citep{munk_nonparametric_2010,reis_asymptotic_2011}, or
wavelets \citep{hoffmann_adaptive_2010}.

In each case, these methods can achieve minimax convergence rates,
equivalent for fixed \(T\) to observing \(c_t\) under Gaussian white
noise of size \(n^{-1/4}.\) In fact, it can be shown this link is a
formal statistical equivalence \citep{reis_asymptotic_2011}.

When jumps are present, however, we must account for them before
estimating \(c_t.\) Methods for doing so include jump thresholding
\citep{mancini_disentangling_2001,mancini_non_2006,fan_multi-scale_2007,jing_estimating_2011},
bipower variation
\citep{barndorff-nielsen_power_2004,podolskij_bipower-type_2009,podolskij_estimation_2009,hautsch_pre-averaging_2012},
and characteristic functions
\citep{todorov_realized_2012,jacod_remark_2012,jacod_efficient_2014}.

Unfortunately, if the jumps are of infinite variation, in general
these methods can no longer achieve the same convergence rates.  Even
given noiseless observations of the efficient prices, it is known that
the minimax convergence rate for \(c_t\) suffers, unless we assume the
infinite-variation part is a scaled \(\beta\)-stable process
\citep{jacod_remark_2012,jacod_efficient_2014}.

Nonetheless, empirical evidence suggests that price data does indeed
contain infinite-variation jumps
\citep{ait-sahalia_estimating_2009,jing_estimating_2011}. In the
following, we will therefore construct a novel estimate of the rate
process \(r_t.\) We will show that our estimate achieves good rates of
convergence in both models, and in the time-changed model is
unaffected by arbitrary jump activity.


Our estimate will be constructed in three stages. We will first obtain
pre-averaged estimates of price increments, and estimates of the
microstructure noise, as in \citet{jacod_microstructure_2009} or
\citet{podolskij_estimation_2009}.

We will then construct local estimates of the spot volatility, derived
by estimating the characteristic function of the price process. While
our approach will be similar to ones considered by previous authors
\citep{todorov_realized_2012,jacod_remark_2012,jacod_efficient_2014},
the precise construction necessary to obtain minimax rates will be
new.

Finally, we will smooth our local estimates of the volatility, using
standard tools from nonparametric regression. While many such
approaches are possible, we will use local polynomials, as described
for example by \citet{tsybakov_introduction_2009}. We will also
discuss how the various parameters required can be chosen
automatically from the data.

We will then prove results on the convergence rates of our
estimates. We note that our results will apply in two settings: a
standard nonparametric setting, where the characteristics of \(X_t\)
are assumed fixed and smooth; and a setting more natural in
quantitative finance, where these characteristics are themselves
described by It\=o semimartingales with locally bounded
characteristics.

For simplicity, our results will focus on the high-frequency case,
where the fixed time horizon \(T = 1.\) We note, however, that similar
results can also be proved when \(T \to \infty,\) provided that the
step size \(T / n \to 0.\)

In the time-changed L\'evy model \eqref{eq:tc}, we will then show that
our procedure estimates \(r_t\) with minimax convergence rates, equal
to those in the Gaussian white noise model with noise level
\(n^{-1/4}.\) Our results will hold under arbitrary jump activity, and
without knowledge of the L\'evy parameters \(b,\) \(c\) and \(\nu.\)

In the general semimartingale model \eqref{eq:sm}, we will show that
our procedure continues to estimate \(r_t.\) 
While lower bounds for this problem are still unknown, the convergence
rates we will obtain are as good as any implied by previous work. Our
estimate thus achieves the best of both worlds: good convergence when
the time-changed approximation is accurate, and no penalty when it is
not.

In \autoref{sec:est}, we will give the construction of our estimates,
and in \autoref{sec:mod}, describe the specific assumptions we
consider. In \autoref{sec:con}, we will then state our results on
convergence rates, and in \autoref{sec:proofs}, give proofs.

\section{Local characteristic-function estimates}
\label{sec:est}

In this section, we will define our estimates of the volatility and
rate processes.  As described in the introduction, our estimates are
constructed in three stages: pre-averaging, spot volatility
estimation, and smoothing.

We begin with the pre-averaging step, and proceed using the
construction of \citet{reis_asymptotic_2011}. We must first subdivide
the time interval \([0,1]\) into a number \(n_0\) of equal bins. To
define \(n_0,\) we choose \(n_1, n_2 \in \N\) in terms of \(n,\) so
that
\[\phantomsection \label{def:n} n_m \sim h_m^{-1} n^{(2m-1)/8},\qquad m=1,2,\]
for bandwidths \(h_1, h_2 > 0,\) and set \(n_0 = n_1n_2.\)

We then divide \([0,1]\) into \(n_0\) bins, and compute on each
one a pre-averaged estimate \(\hxk\) of the increments of \(X_t.\) We
will compute \(\hxk\) by integrating the observed increments against a
scaling function \(\Phi_n(t);\) we define
\[\phantomsection \label{def:phi} \Phi_n(t) = \sqrt{n_0}\Phi(n_0t), \qquad \Phi(t) = 2 \sin(2\pi t).\]

The specific choice of scaling function \(\Phi_n\) is motivated by
\citet{reis_asymptotic_2011}, who shows that in a Gaussian setting,
functions of this form are most efficient at extracting information
from noisy data.  We note that our choice of \(\Phi\) includes a full
period of the sine wave in each bin, rather than a half period; this
choice allows us to ensure that the pre-averaged increments are
approximately symmetrically distributed, a property we will require
when modelling the behaviour of infinite-variation jumps.

We may now define the pre-averaged increments \(\hxk.\) For \(k = 0,
\dots, n_0-1,\) define index sets \(J_k = (n/n_0)[k, k+1) \cap \Z,\)
and let
\[\phantomsection \label{def:hx} \hxk = \sum_{j,j+1 \in J_k} \aj (Y_{j+1} - Y_j), \qquad \aj =
\Phi_n(j/n).\] The estimate \(\hxk\) thus averages the observed
increments of \(X_t\) over the time interval \([k,k+1)/n_0.\) We can
also define an estimate \(\hssk\) of the microstructure noise over the
interval. We set
\[\phantomsection \label{def:hs} \hssk = \frac{n_0}{2n}\sum_{j,j+1 \in J_k}
(Y_{j+1}-Y_{j})^2,\] proportional to the realised quadratic variation
of the observations.

We next describe our spot volatility estimation step. We will
subdivide \([0,1]\) into \(n_2\) larger bins, and on each one,
construct an estimate \(\hclu\) of the volatility \(c_t.\) While our
approach will be based on local characteristic function estimates,
similar to those considered by previous authors
\citep{todorov_realized_2012,jacod_remark_2012,jacod_efficient_2014},
the precise construction necessary to obtain minimax rates will be
new.

For \(l = 0, \dots, n_2-1,\) we define index sets \(\kl =
n_1[l,l+1) \cap \Z,\) and local estimates \(\hvlu\) of the
characteristic function of increments of \(X_t,\) given by
\[\phantomsection \label{def:hv} \hvlu = \frac1{n_1}\sum_{k \in K_l}
\cos(u\hxk).\] We note that \(\hvlu\) thus averages the 
cosines of the \(\hxk\) over the time interval \([l,l+1)/n_2.\) We
can also define an estimate \(\hplu\) of the corresponding
contribution of the microstructure noise; we set
\[\phantomsection \label{def:hp} \hplu = \frac1{n_1}\sum_{k \in K_l} \exp(-\kappa u^2 \hssk),
\qquad \kappa = \frac{4\pi^2 n_0^2}{n}.\]

If the log-prices \(X_t\) and noises \(\ej\) were Gaussian, then by
considering their characteristic functions, we would expect
\[\hvlu \approx \exp(-(c_{l/n_2}+\kappa \sigma_{l/n_2}^2)u^2), \qquad
\hplu \approx \exp(-\kappa\sigma_{l/n_2}^2u^2).\] 
We could then rearrange these quantities to obtain an estimate
\[-\frac1{u^2}\log\,\abs*{\frac{\hvlu}{\hplu}}\] of the volatility
\(c_{l/n_2}.\)

In fact, such an estimate would be biased. We can, however, provide
bias-corrected estimates \(\hclu\) of \(c_{l/n_2};\) we define
\[\phantomsection \label{def:hc} \hclu = -\frac1{u^2} \left( \log \,\abs*{\frac{\hvlu}{\hplu}} +
  \frac\hulu2\right),\] where the bias-correction term
\[\phantomsection \label{def:ht} \hulu = \frac1{n_1} \left(\frac{1 + \widehat
    \varphi_l(2u)}{2\hvlu^2} - 1\right).\]

We have thus defined an estimate \(\hclu\) of the spot volatility.
The advantage of this procedure over other such estimates is that it
naturally accounts for the presence of jump activity: we will show
that, for general semimartingales, \(\hclu\) is an
asymptotically-unbiased estimate of the quantity \(\clu,\) given by
the adjusted volatility process
\[\phantomsection \label{def:cu} \ctu = c_t + \frac{1}{n_0u^2}\int_0^1 \int_\R
(1-\cos(\sqrt{n_0}\Phi(w)ux))\, \nu_t(dx)\,dw.\]

The process \(\ctu\) thus includes both the volatility \(c_t,\) and a
term depending on the jump measure \(\nu_t.\) As \(n\to \infty,\) the
term involving \(\nu_t\) vanishes; however, when the jump activity
\(\beta\) is large, this term will not vanish fast enough to be
negligible, and so we must consider it explicitly.  Crucially in the
time-changed model~\eqref{eq:tc}, we have that both terms enter
\(\ctu\) linearly, and so \(\ctu\) is proportional to the rate process
\(r_t.\)

In either model, since \(r_t\) integrates to one, we may estimate it
by normalising our estimates of \(\ctu.\) However, to estimate \(r_t\)
optimally we will not be able to use the preliminary estimates
\(\hclu\) directly, as their variance is too large. First, we must
smooth them, using standard tools from nonparametric regression.

While many such approaches are possible, in the following we will use
a local polynomial estimate of \(\ctu,\) as described for example by
\citet{tsybakov_introduction_2009}.  To define our estimate, fix a
non-negative kernel function \(K : \R \to \R,\) supported on
\([-1,1],\) and satisfying \(\int_\R K(t) \,dt = 1.\) Also fix an
order \(N \in \N,\) and bandwidth \(h > 0.\) Then let \(\tctu\) denote
a local polynomial estimate of \(\ctu\) of degree \(N-1,\) using the
observations \(\hclu,\) kernel \(K,\) and bandwidth \(h.\)

In other words, let
\[\phantomsection \label{def:lp} \tctu = \sum_{l=0}^{n_2-1} W_{n,l}(t)\hclu,\]
where the weight functions \(W_{n,l}(t)\) are given by
\[\phantomsection \label{def:lpw} W_{n,l}(t) = \frac1{n_2h} K(\lambda_{n,l}(t)) U(0)^T V_{n}(t)^{-1} U(\lambda_{n,l}(t)),\]
for the terms
\begin{align*}
\lambda_{n,l}(t) &= \frac1h\left(t-\frac{l}{n_2}\right),\\
U(\lambda) &= \left(1, \lambda, \dots, \frac{\lambda^{N-1}}{(N-1)!}\right)^T,\\
V_{n}(t) &= \frac{1}{n_2h} \sum_{l=0}^{n_2-1} K(\lambda_{n,l}(t))
U(\lambda_{n,l}(t)) U(\lambda_{n,l}(t))^T.
\end{align*}

The constant \(N \in \N\) serves as an upper bound for the smoothness
we expect of the volatility process \(c_t,\) and other processes
related to \(X_t.\) We include here the case where \(N\) is large, so
that our estimate can match known nonparametric lower bounds for a
wide range of smoothness.

In practice, however, we may believe that these processes are It\=o
semimartingales, as in most common financial models. We will see later
that in this case it suffices to take \(N=1;\) the above estimate
then reduces to the Nadaraya-Watson kernel estimate, with weights
\(W_{n,l}(t)\) given by
\[W_{n,l}(t) =
\frac{K(\lambda_{n,l}(t))}{\sum_{l=0}^{n_2-1} K(\lambda_{n,l}(t))}.\]

In either case, we then have an estimate \(\tctu\) of the volatility
\(c_t.\) To estimate the normalised volatility or rate process
\(r_t,\) we likewise define the normalised estimate
\[\phantomsection \label{def:hr} \trtu = \frac{\tctu}{\frac1{n_2}\sum_{m=0}^{n_2-1}\hcmu} =
\sum_{l=0}^{n_2-1}W_{n,l}(t)\hrlu,\] where \[\hrlu =
\frac{\hclu}{\frac1{n_2}\sum_{m=0}^{n_2-1}\hcmu}.\]

In the following sections, we will prove results on the theoretical
performance of our estimates \(\tctu\) and \(\trtu.\) We first,
however, briefly discuss their implementation. In particular, we note
that the above estimates require the choice of a number of parameters:
the kernel \(K,\) order \(N,\) frequency \(u,\) and bandwidths
\(h_1,\,h_2\) and \(h.\)

In general, good performance in nonparametric regression can be
obtained with a range of kernels \(K;\) popular choices include the
uniform, Epanechnikov and biweight kernels, given by
\(\mathrm{Beta}(k,k)\) densities for \(k = 1, 2\) and 3 respectively. If
we believe the volatility \(c_t,\) and other characteristic processes,
are given by It\=o semimartingales, then as noted above, we may also
take \(N=1.\)

The remaining parameters \(u,\,h_1,\,h_2,\) and \(h\) are more
important. We will show in the following that the variances of our
estimates \(\hclu\) depend crucially on the choice of the frequency
\(u,\) and bandwidths \(h_1,h_2.\) The correct choice of the bandwidth
\(h\) is likewise known to be crucial generally in nonparametric
regression.

To select these parameters, we can borrow methods from nonparametric
statistics. While many such approaches are available, we will briefly
mention the heuristic of generalised cross-validation, a popular
method for choosing the bandwidth \(h\) in nonparametric regression
\citep{golub_generalized_1979}.

The GCV criterion
\[GCV(u,h_1,h_2,h) = \frac{\frac1{n_2} \sum_{l=0}^{n_2-1}
  (\trlu-\hrlu)^2}{\left(\frac1{n_2}\sum_{l=0}^{n_2-1}W_{n,l}(l/n_2)\right)^2}\]
provides an estimate of the \(L^2\) error in \(\trtu.\) We can then
choose \(u,\,h_1,\,h_2\) and \(h\) to minimise this criterion, using
any standard global optimisation algorithm.

Simple tests on simulated data show that minimising this criterion
provides sensible choices of the parameters, for both estimates
\(\tctu\) and \(\trtu.\) (We note that it is inadvisable to apply GCV
to \(\tctu\) directly, as the criterion then favours parameters which
shrink the estimate to zero.)

We have thus described new estimates of the volatility \(c_t,\) and
normalised volatility or rate process \(r_t;\) however, we have yet to
consider their performance.  In the following sections, we will show
that, for suitable choices of the parameters, these estimates can
obtain good rates of convergence over both the semimartingale model
\eqref{eq:sm}, and time-changed L\'evy model \eqref{eq:tc}.

\section{Semimartingale and L\'evy models} 
\label{sec:mod}

In this section, we will describe the assumptions we make on our
data. Our assumptions will be satisfied by common models in both
nonparametric statistics and quantitative finance. Under these
assumptions, we may then proceed to show that our estimates \(\tctu\)
and \(\trtu\) achieve good rates of convergence.

\phantomsection \label{def:time}
We first assume that the log-prices \(X_t\) are generated under
the general It\=o semimartingale model \eqref{eq:sm}, and our
observations \(Y_j\) come from the microstructure noise model
\eqref{eq:ms}, with fixed time horizon \(T = 1.\) For simplicity, we
do not consider further other choices of \(T,\) but we note that
similar results can also be proved when \(T \to \infty,\) \(T/n \to
0.\)

\phantomsection \label{def:fp}
Our assumptions will then be stated in terms of a filtration \(\ft,\)
\(t \in [0,1],\) with respect to which the semimartingale
decomposition \eqref{eq:sm} and zero-mean condition of \eqref{eq:ms}
hold.  As in \citet{jacod_microstructure_2009}, to allow for the
modelling of microstructure noise, we will not assume that the
filtration \(\ft\) is right-continuous. Instead, we will require that
the semimartingale decomposition \eqref{eq:sm} is also valid with
respect to the filtration \(\ftp = \bigcap_{s >t} \mathcal F_s,\) and
that the noises \(\ej\) are \(\fjp\)-measurable.

\phantomsection \label{def:ss}
We then let \(\ms\) denote the class of probability measures \(\P\)
satisfying the above conditions, on some filtered measurable space
\((\Omega, \mathcal F, \ft).\) We will also make some further
assumptions on the characteristics \(b_t,\) \(c_t\) and \(\nu_t,\) and
errors \(\ej.\)

We begin by defining a smoothness assumption on \(\ft\)-adapted
processes. We will require in the following that the volatility
\(c_t,\) and other characteristic processes of the log-prices \(X_t\)
and noises \(\ej,\) satisfy this assumption with high probability.

\begin{definition}
  \label{def:iam}
  Let \(S \in [0,1]\) be an \(\ft\)-stopping time, \(\alpha \ge \oh,\)
  \(D > 0,\) and set \(\alpha_0 = 1 \wedge \alpha.\) We define
  \(\iam\) to be the class of \(\ft\)-adapted complex-valued processes
  \(Z_t,\) for which the stopped process \(Z_{t\wt}\) satisfies:
  \begin{enumerate}
  \item \(\abs{Z_{t\wt}} \le D, \, t \in [0,1];\)
  \item \(\E\left[\abs{Z_{s\wt}-Z_{t\wt}}^2\mid\ftp\right] \le
    D^2(s-t)^{2\alpha_0}, \, 0 \le t \le s \le 1;\) and
  \item if \(\alpha > 1,\) then letting \(m\) denote the largest
    integer smaller than \(\alpha,\) \(Z_{t\wt}\) has \(m\)-th real
    derivative \(Z_{t\wt}^{(m)}\) satisfying
    \[\E\left[\abs{Z_{s\wt}^{(m)}-Z_{t\wt}^{(m)}}^2\mid\ftp\right] \le
    D^2(s-t)^{2(\alpha-m)}, \quad 0 \le t \le s \le 1.\]
  \end{enumerate}
\end{definition}

The classes \(\iam\) thus contain all processes \(Z_t\) which, when
stopped by \(S,\) are bounded and smooth in quadratic mean. We note
that these classes describe a variety of processes. Firstly, the
classes \(\iom\) contain all processes which are almost-surely
\(\alpha\)-H\"older, with constant \(D.\) More generally, the
following lemma shows that the classes \(\ihm\) can describe all
c\`agl\`ad It\=o semimartingales with locally bounded characteristics.

\begin{lemma}
  \label{lem:ihm}
  Let \(Z_t\) be a c\`agl\`ad It\=o semimartingale, having
  decomposition
  \begin{multline*}
    Z_t = Z_0 + \int_0^{t^-} b_{Z,s}\,ds + \int_0^{t^-}
    \sqrt{c_{Z,s}}\,dB_{Z,s}\\ + \int_0^{t^-} \int_\R x
    \,(\mu_Z(dx,ds) - 1_{\abs{x} <
      1}\,\nu_{Z,s}(dx)\,ds)\end{multline*} with respect to both
  filtrations \(\ft\) and \(\ftp.\) Suppose the processes \(b_{Z,t}\)
  and \(c_{Z,t}\) are locally bounded, as are the processes
  \(\int_{\abs{x} < R} x^2\,\nu_{Z,t}(dx)\) for all \(R > 0.\) Then
  for each \(D > 0,\) there exists an event \(\te \in \fz,\) and
  \(\ft\)-stopping time \(S,\) such that on \(\te,\) \(Z_t \in \ihm,\)
  with \(\P(\te \cap \{S = 1\}) \to 1\) as \(D \to \infty.\)
\end{lemma}

We now define our assumptions on the observations \(Y_j.\)
Essentially, our results will be proved in models where with high
probability, the drift \(b_t\) is bounded, and the stochastic
volatility \(c_t,\) jump process \(\nu_t,\) and noise variance
\(\sigma_t^2\) are bounded and smooth.

\begin{definition}
  \label{def:scm}
  Let \(\alpha \ge \oh,\) \(\beta \in [0,2],\) \(\gamma \in [0,1],\)
  and \(0 < C \le D.\) We define \(\saecm\) to be the class of
  probability measures \(\P \in \ms\) which satisfy the following
  conditions, on an event \(\te \in \fz,\) and for a stopping time \(S
  \in [0,1],\) with \(\P(\te \cap \{S = 1\}) \ge 1 - \gamma.\)
  \begin{enumerate}
  \item The noises \(\ej\) have variance
    \[\E[\ej^2 \mid \fj] = \sigma^2_{j/n},\]
    for a latent process \(\sigma^2_t \in \iam;\) and have bounded fourth moment,
    \[\E[\ej^4 \mid \fj] \le D^2, \quad j/n \le S.\]
  \item
    The drift process \(b_t\) is bounded,
    \[\abs{b_t} \le D, \quad t \in [0,S].\]
  \item The volatility process \(c_t \in \iam,\) and is also bounded
    below,
    \[c_t \ge C, \quad t \in [0,S].\]
  \item The jump activity is of index at most \(\beta,\)
    \[\int_\R (1 \wedge \abs{x}^\beta) \,\nu_t(dx) \le D, \quad t
    \in [0,S];\] and if \(\beta > 1,\) for any measurable function
    \(f:\R \to \C\) with \[\abs{f(x)} \le (1 \wedge x^2),\] we have
    \[\int_\R f(x)\,\nu_t(dx) \in \iam.\]
\end{enumerate}
  \phantomsection \label{def:sc}
Also define \(\sacm = \sazcm,\) and let \(\sa\) denote the class of
probability measures \(\P\) which lie in some \(\saecm\) for each
\(0 < C \le D,\) with \(\gamma \to 0\) as \(C \to 0,\) \(D \to
\infty.\)
\end{definition}

In the following, our theoretical results will be first proved for the
classes \(\sacm,\) where the characteristics of the log-price \(X_t\)
and noises \(\ej\) are almost-surely bounded and smooth. These classes
will be the most convenient for our analysis, and will allow us to
draw comparisons with previous nonparametric results in the
literature.

We note that these classes impose quite strong conditions on our
process \(X_t;\) in particular, they require the volatility \(c_t\) to
be bounded away from zero. However, we will also generalise our
results to the larger classes \(\sa,\) which only require these
conditions to hold locally; in particular, they only impose the weaker
bound that \(c_t > 0\) almost-surely.

The parameters \(\alpha\) and \(\beta\) govern two
different smoothness properties of \(X_t.\) The parameter \(\alpha\)
measures the smoothness of the characteristics of \(X_t\) and the
\(\ej\) over time: if \(X_t\) is a L\'evy process, and the \(\ej\)
have constant variance, the above conditions can hold for any value of
\(\alpha.\) In contrast, the parameter \(\beta\) governs the jump
activity of \(X_t,\) and thus also the smoothness of its sample
paths.

We note that in typical semimartingale models, the parameter
\(\alpha=\tfrac12;\) we include here the case \(\alpha > \tfrac12\) to
allow for comparison with previous results in nonparametrics. We also
note that, since the smoothness \(\alpha\) is measured in square mean,
the case \(\alpha = \tfrac12\) allows for jumps in the volatility, or
in other characteristics; it thus allows the characteristics to
depend smoothly on \(X_t,\) for any level \(\beta\) of jump
activity. More generally, \autoref{lem:ihm} shows that the classes
\(\sh\) contain most common models for financial processes.

As some of our results will be specific to the time-changed
model~\eqref{eq:tc}, we additionally define submodels describing this
case. We note that our definition includes a choice of normalisation;
as in \eqref{eq:def-r}, we assume the rate process \(r_t\) integrates
to one.

\begin{definition}
\label{def:ta}
Let \(\mt\) denote the class of probability measures \(\P \in \ms\)
satisfying \eqref{eq:tcsm}, for a drift \(b \in \R,\) volatility \(c >
0,\) L\'evy measure \(\nu,\) and rate process \(r_t > 0\) given by
\eqref{eq:def-r}.
Also define the models \(\tacm = \satcm \cap \mt,\) and \(\ta = \sat
\cap \mt.\)
\end{definition}

This choice of normalisation is most convenient for our results, but
we note that others are also possible; for example, we might prefer
the deterministic normalisation \(\E[\int_0^T r_s\,ds] = T.\) If we
made an ergodicity assumption on the process \(r_t,\) as in
\citet{figueroa-lopez_nonparametric_2009}, then for suitable \(T \to
\infty,\) \(T/n \to 0,\) we would have that \(\int_0^{T} r_s\,ds\) is
close to \(\E[\int_0^{T}r_s\,ds].\) The two normalisations would then
also be close, and our arguments would apply equally in either
setting.

In the following, for simplicity, we will concentrate on
\autoref{def:ta}.  We then have that in particular, the class \(\toh\)
covers most common financial models for time-changed L\'evy
processes. With these definitions, we are now ready to give our
results on the performance of our estimates.

\section{Convergence results}
\label{sec:con}

In this section, we will show that our estimates \(\tctu\) and
\(\trtu\) have good rates of convergence, in both the general
semimartingale models \(\sa,\) and time-changed L\'evy models \(\ta.\)
In particular, we will establish that in \(\ta,\) the time-change
\(r_t\) can be recovered at minimax rates under arbitrary jump
activity.

We first define some additional processes which will be relevant to
our results. We set
\[\phantomsection \label{def:vt}
\vtu = \exp(-c_t(u)u^2)\ptu, \qquad \ptu = \exp(-\kappa \sigma^2_t
u^2),\]
processes we will show describe the means of the estimates \(\hvlu\)
and \(\hplu.\) We also set
\[\phantomsection \label{def:tt}
\rho_t^2(u) = \tfrac12(1 + \varphi_t(2u)) - \varphi_t^2(u), \qquad
\tau_t^2(u) = \rho_t^2(u)/n_1\varphi_t^2(u),\]
processes we will show describe the variances of the estimates
\(\hvlu\) and \(\hclu.\)

We now begin with a result on the accuracy of the preliminary
estimates \(\hclu.\) At this stage, our results will be proved solely
in the bounded semimartingale model \(\sacm;\) we will return later to
the consequences for our other models.

We can establish that, on events with high probability, our
preliminary estimates \(\hclu\) have asymptotic mean \(\clu,\) and
variance \(\ulu.\) We can further show that the errors in these statements
are of order \(n^{-\alpha_1}\) and \(n^{-\alpha_2}\) respectively,
where the rates
\[\phantomsection \label{def:ao}
\alpha_1 = \frac14 \wedge \frac{3\alpha}8,\qquad\alpha_2 =
\frac{\alpha_1}2 + \frac1{16}.\]

\begin{theorem}
  \label{thm:est}
  Fix \(u,h_1,h_2 > 0,\) \(\alpha \ge \oh,\) \(\beta \in [0,2],\) \(0
  < C \le D,\) and suppose \(\P \in \sacm.\) Then the local
  volatility estimates \(\hclu\) are \(\fls\)-measurable, and we have
  events \(E_l \in \fls,\) satisfying
  \[\P(E_l^c \mid \fl) \le \exp(-A n^{1/8})\]
  for a constant \(A > 0,\) on which
  \begin{align*}
    \E[(\hclu - \clu)1(E_{l}) \mid \fl] &= O(n^{-\alpha_1}),\\
    \E[(\hclu-\clu)^21(E_{l}) \mid \fl] &= \ulu + O(n^{-\alpha_2}).
  \end{align*}
  Furthermore, these results are uniform over \(l = 0, \dots, n_2-1,\)
  and \(\P \in \sacm.\)
\end{theorem}

We thus have that the estimates \(\hclu\) behave roughly like
\(n^{3/8}\) observations of the adjusted volatility process \(\ctu,\)
under errors with variance \(n^{-1/8}.\) In other words, we obtain an
accuracy like observing the process \(\ctu\) under \(n^{-1/4}\) white
noise. While our estimates \(\hclu\) also include an additional bias
term, and are accurate only on a set of high probability, we will
nonetheless see that they are good enough to accurately recover the
volatility \(c_t\) or time-change \(r_t.\)

We now establish that our regression estimate \(\tctu\) is a good
estimate of the adjusted volatility \(\ctu.\) We will measure the
accuracy of our estimates both pointwise, and in the \(L^2\)-norm,
\[\norm{f}^2_2 = \int_0^1 f(t)^2\,dt.\]
In these metrics, we will show that \(\ctu\) can be recovered at the
rate \(n^{-\alpha_3},\) where \[\phantomsection \label{def:at} \alpha_3 =
\frac\alpha{2(2\alpha+1)}\] is the standard minimax rate for recovering
a function of smoothness \(\alpha\) under \(n^{-1/4}\) white noise.

\begin{theorem}
  \label{thm:reg}
  Fix a kernel \(K\) as in \autoref{sec:est}, \(N \in \N,\) \(u \in
  \R,\) \(h_1, h_2 > 0,\) \(\alpha \in [\oh, N],\) \(\beta \in
  [0,2],\) \(0 < C \le D,\) let \(h \sim n^{-1/2(2\alpha+1)},\) and
  suppose \(\P \in \sacm.\) We then have an event \(E,\) satisfying
  \[\P(E^c \mid \fz) \le \exp(-An^{1/8})\]
  for a constant \(A > 0,\) on which
  \begin{align*}
    \E[\abs{\tctu - \ctu}^21(E) \mid \fz]^{1/2}&= O(n^{-\alpha_3}),\\
    \intertext{uniformly in \(t \in [0,1],\) and}
    \E[\norm{\tcu - c(u)}_2^21(E) \mid \fz]^{1/2} &= O(n^{-\alpha_3}).
  \end{align*}
  Furthermore, these results are uniform over \(\P \in \sacm.\)
\end{theorem}

We thus have that the regression estimates \(\tctu\) accurately
recover \(\ctu,\) in the model \(\sacm.\) It remains to deduce
consequences for the volatility \(c_t\) and time-change \(r_t,\) in
the more general models \(\sa\) and \(\ta.\) In \(\sa,\) we will
obtain the rate \(n^{-\alpha_4},\) where
\[\phantomsection \label{def:af} \alpha_4 = \alpha_3 \wedge \frac{2 - \beta}4\] depends also on
the jump activity \(\beta\) of the log-price \(X_t.\) When estimating
\(r_t\) in \(\ta,\) however, we will retain the convergence rate
\(n^{-\alpha_3},\) even under arbitrary jump activity.

\begin{corollary}
  \label{cor:reg}
  Let the parameters
  \(K,\,N,\,u,\,h_1,\,h_2,\,\alpha,\,\beta,\,C,\,D\) and \(h\) be as
  in \autoref{thm:reg}.
  \begin{enumerate}
  \item If \(\P \in \sa,\) the estimates \(\tctu\) and \(\trtu\) satisfy
    \begin{align*}
      \abs{\tctu-c_t},\,\abs{\trtu-r_t} &= O_p(n^{-\alpha_4}),\\
      \intertext{uniformly in \(t \in [0,1],\) and}
      \norm{\tcu-c}_2,\,\norm{\tru-r}_2 &= O_p(n^{-\alpha_4}).
    \end{align*}
    Furthermore, these results are uniform over \(\P \in \sacm.\)
  \item If also \(\P \in \mt,\) the results for \(\trtu\) hold with
    improved convergence rate \(n^{-\alpha_3}.\)
  \end{enumerate}
\end{corollary}

We note that convergence does depend on the choice of parameters
\(K,\) \(N,\) \(u,\) \(h_1,\) \(h_2\) and \(h,\) and in particular
requires the bandwidth \(h\) to be chosen as in
\autoref{thm:reg}. Adaptive results in this setting are possible, for
example applying Lepski's method to choose \(h,\) and using Azuma's
inequality to control the deviations in \(\hvlu\) and \(\hplu\)
\citep{lepski_optimal_1997}. For simplicity, however, we will treat
these parameters as fixed, noting that they can be chosen
heuristically as in \autoref{sec:est}.

For the time-changed L\'evy model \(\ta,\) as a simple consequence of
results in \citet{munk_lower_2010}, we can further show that our rates
are optimal. We can likewise provide a partially matching lower bound
for the general semimartingale model \(\sa.\)

\begin{theorem}
  \label{thm:lb}
  Let \(\alpha > \oh,\) \(\beta \in [0,2],\) and \(0 < C < D.\)
  \begin{enumerate}
  \item No estimate \(c_t^*\) of \(c_t\) can satisfy
    \[\norm{c^*-c}_2 = o_p(n^{-\alpha_3}),\] uniformly
    over
    \(\P \in \sacm \cap \mt,\) or
    \[\abs{c_t^* - c_t} = o_p(n^{-\alpha_3}),\]
    uniformly over \(t \in [0,1]\) and \(\P \in \sacm \cap \mt.\)
  \item The same results hold for any estimate \(r_t^*\) of \(r_t.\)
  \end{enumerate}
\end{theorem}

In the general semimartingale model \(\sa,\) if \(\beta\) is large, we
have \(\alpha_3 > \alpha_4,\) and matching lower bounds are more
difficult to establish. We note, however, that our estimates \(\tctu\)
and \(\trtu\) already obtain rates as good as those implied by
previous work under noise.  Furthermore, the recent paper of
\citet{jacod_remark_2012} on the noiseless problem suggests that the
rate \(n^{-\alpha_4}\) is indeed optimal, up to log factors.

It may at first be surprising that the results for \(r_t\) in the
time-changed model \(\ta\) are better than in the general
semimartingale model \(\sa,\) when the jump activity is
large. However, we know that the difficulty in estimating the
volatility \(c_t\) in \(\sa\) comes primarily from distinguishing
\(c_t\) and \(\nu_t.\) We obtain improved convergence rates in \(\ta\)
because in this model, we can estimate the rate process \(r_t\)
without having to separate \(c_t\) and \(\nu_t.\)

We have thus shown that our estimate \(\trtu\) can recover the
time change in a noisy L\'evy model at the minimax rate, equivalent to
observing \(r_t\) under \(n^{-1/4}\) white noise. It can do so without
knowledge of the distribution of the L\'evy process, and under
arbitrary jump activity.

Furthermore, in the general semimartingale setting, where the L\'evy
assumption may be violated, \(\trtu\) remains a valid estimate of the
normalised volatility. In this setting, we again achieve good rates,
governed either by the noise level of \(n^{-1/4},\) or by a bias due
to jump activity, common to all volatility estimates.

\section{Proofs}
\label{sec:proofs}

We now give proofs of our results. We prove results on the preliminary
estimates \(\hclu\) in \autoref{sec:proofs-pre}, and results on
convergence rates in \autoref{sec:proofs-con}. Technical
proofs\notlink\ are given in \applink.

\subsection{Proofs on preliminary estimates}
\label{sec:proofs-pre}

We first prove \autoref{thm:est}, our result bounding the error in our
preliminary estimates \(\hclu.\) Our proof will use a series of
lemmas, controlling the behaviour of the various components of
\(\hclu.\) We begin by stating some technical lemmas; proofs are given
in \applink.

\begin{lemma}
  \label{lem:tech2}
  In the setting of \autoref{thm:est}, fix \(u \in \R,\) and let
  \(\xi_t\) denote (i) \(c_t(u),\) (ii) \(\varphi_t(u),\) or (iii)
  \(\psi_t(u).\) In each case, for \(n \in \N,\) \(0 \le t \le s \le
  1,\) we have
  \[\E[(\xi_s(u) - \xi_t(u))^2 \mid \ftp] =
  O((s-t)^{2\alpha_0} + n^{-1/2}).\] Furthermore, we have (iv)
  \(c_t(u) \le 3D,\) almost surely.
\end{lemma}

\begin{lemma}
  \label{lem:xcf}
  In the setting of \autoref{thm:est}, for \(k = 0, \dots, n_0-1,\)
  and \(u \in \R,\) we have
  \begin{align*}
    \E[\cos(u\hxk) \mid \fk]
    &= \vku + O(n^{-1/4}),\\
    \Var[\cos(u\hxk) \mid \fk]
    &= \rho_{k/n_0}^2(u) + O(n^{-1/4}).
  \end{align*}
\end{lemma}

\begin{lemma}
  \label{lem:hssk}
  In the setting of \autoref{thm:est}, for \(k = 0, \dots, n_0-1,\)
  and \(u \in \R,\) we have
  \begin{align*}
    \E[\exp(-\kappa \hssk u^2)\mid\fk]&= \pku + O(n^{-1/4}),\\
    \Var[\exp(-\kappa \hssk u^2)\mid\fk]&= O(n^{-1/4}).
  \end{align*}
\end{lemma}  

We are now in a position to describe the behaviour of the estimates
\(\hvlu\) and \(\hplu.\) First, we will define the event \(\el\)
mentioned in the statement of \autoref{thm:est}. We set
\[\phantomsection \label{def:el} \el = \{\hvlu \ge \zeta(u)\} \cap \{\hplu \ge
\zeta(u)\},\]
where the constant \(\zeta(u) = \oh \exp(-(\kappa+3)Du^2).\) We then
have the following result.

\begin{lemma}
  \label{lem:bounds}
  In the setting of \autoref{thm:est}, for \(l = 0, \dots, n_2-1,\) we
  have:
  \begin{enumerate}
    \item \(\E[\hvlu-\vlu\mid\fl] = O(n^{-\alpha_1});\)
    \item \(\E[\hplu-\plu\mid\fl] = O(n^{-\alpha_1});\)
    \item \(\E[(\hvlu-\vlu)^2\mid\fl] = \tlu/n_1 + O(n^{-\alpha_1});\)
    \item \(\E[(\hplu-\plu)^2\mid\fl] = O(n^{-1/4});\)
    \item for \(p = 3, 4,\) \(\E[(\hvlu - \vlu)^p \mid \fl] =
      O(n^{-\alpha_1});\) and
    \item \(\P(\elc \mid\fl)
      \le \exp(-An^{1/8}),\) for a constant \(A > 0.\)
  \end{enumerate}
\end{lemma}

\begin{proof}
  We first note that
  \[\hvlu - \vlu = \frac1{n_1}\sum_{k\in K_l}Z_{\delta,k},\]
  where the random variables
  \[Z_{\delta,k} = \cos(u\hxk) - \vlu,\quad k \in K_l;\] we will begin
  by proving some facts about the \(Z_{\delta,k}.\)
  We have \(\abs{Z_{\delta,k}} \le 2,\) and
  \begin{align}
    \notag &\E[\E[Z_{\delta,k} \mid \fk]^2\mid\fl]\\
    \notag &\qquad = \E[(\vku-\vlu + O(n^{-1/4}))^2\mid\fl],\\
    \intertext{using \autoref{lem:xcf},}
    \notag &\qquad = O(1)\E[(\vku-\vlu)^2\mid\fl] + O(n^{-1/2})\\
    \label{eq:zs} &\qquad = O(n^{-2\alpha_1}),
  \end{align}
  using \autoref{lem:tech2}(ii).

  We also have
  \begin{align*}
    &\E[Z_{\delta,k}^2\mid\fk]\\
    &\qquad = \Var[Z_{\delta,k}^2\mid\fk] +
    \E[Z_{\delta,k}\mid\fk]^2\\
    &\qquad = \tku + (\vku - \vlu)^2 + O(n^{-1/4}),
  \end{align*}
  using \autoref{lem:xcf}, so
  \begin{align}
    \notag
    &\E[(\E[Z_{\delta,k}^2\mid\fk] - \tlu)^2 \mid \fl]\\
    \notag &\qquad = O(1)\E[(\tku - \tlu)^2 + (\vku
    - \vlu)^2\mid \fl]\\
    \notag &\qquad \qquad + O(n^{-1/2})\\
    \label{eq:zsv} &\qquad = O(n^{-2\alpha_1}).
  \end{align}
  using \autoref{lem:tech2}(ii). We may now prove the claims of the theorem.
  \begin{enumerate}
  \item We have
    \begin{align*}
      \E[\hvlu-\vlu\mid\fl] &= \frac1{n_1}\sum_{k\in\kl}\E[Z_{\delta,k} \mid
\fl]\\
&=
\frac{O(1)}{n_1}\sum_{k\in\kl}\E[\abs{\E[Z_{\delta,k}\mid\fk]}\mid\fl]\\
&=O(n^{-\alpha_1}),
\end{align*}
using \eqref{eq:zs}.
\item The result follows similarly to (i), using
  \autoref{lem:tech2}(iii) and \autoref{lem:hssk}.
\item We have
    \begin{align*}
      \E[Z_{\delta,k}^2\mid\fl]
      &= \tlu + \E[\E[Z_{\delta,k}^2\mid\fk] - \tlu\mid\fl]\\
      &= \tlu + O(n^{-\alpha_1}),
    \end{align*}
    using \eqref{eq:zsv}.
    Likewise, for \(k, k_1 \in K_l,\) \(k > k_1,\) we have
    \begin{align*}
      \E[Z_{\delta,k}Z_{\delta,k_1}\mid\fl]
      &= \E[\E[Z_{\delta,k}\mid\fk]Z_{\delta,k_1}\mid\fl]\\
      &= O(1)\E[\abs{\E[Z_{\delta,k} \mid \fk]}\mid\fl]\\
      &= O(n^{-\alpha_1}),
    \end{align*}
    using \eqref{eq:zs}.
    We deduce that
    \begin{align*}
      &\E[(\hvlu-\vlu)^2 \mid \fl]\\
      &\qquad = \E\left[\frac1{n_1^2} \sum_{k \in K_l}
      Z_{\delta,k}^2 + \frac{2}{n_1^2}\underset{k > k_1}{\sum_{k,k_1 \in K_l,}}
      Z_{\delta,k}Z_{\delta,k_1}\mid\fl\right]\\
      &\qquad = \tlu/n_1 + O(n^{-\alpha_1}).
    \end{align*}
  \item For \(k \in \kl,\) by a similar argument, we have
    \[\E[(\exp(-\kappa \hssk u^2)-\plu)^2 \mid \fl]
      = O(n^{-1/4}),\]
    using \autoref{lem:hssk}. The result follows.
  \item
    We first consider the case \(p = 3.\) For \(k \in K_l,\) we have
    \[\E[Z_{\delta,k}^3 \mid \fl] = O(1),\]
    and for \(k, k_1 \in K_l,\) \(k
    > k_1,\)
    \begin{align*}
      &\E[Z_{\delta,k}^2Z_{\delta,k_1}\mid \fl]\\
      &\qquad = \E[\E[Z_{\delta,k}^2 \mid \fk]Z_{\delta,k_1} \mid \fl]\\
      &\qquad = \tlu\E[Z_{\delta,k_1} \mid \fl]\\
      &\qquad \qquad + O(1)\E[\abs{\E[Z_{\delta,k}^2 \mid \fk] - \tlu}
      \mid \fl]\\
      &\qquad = O(n^{-\alpha_1}),
    \end{align*}
    using \eqref{eq:zs} and \eqref{eq:zsv}.

    Similarly, for \(k, k_1, k_2 \in K_l,\) \(k > k_1, k_2,\) we have
    \begin{align*}
      \E[Z_{\delta,k}Z_{\delta,k_1}Z_{\delta,k_2} \mid \fl]
      &= \E[\E[Z_{\delta,k} \mid \fk]Z_{\delta,k_1}Z_{\delta,k_2}\mid\fl]\\
      &= O(1)\E[\abs{\E[Z_{\delta,k} \mid \fk]}\mid\fl]\\
      &= O(n^{-\alpha_1}),
    \end{align*}
    using \eqref{eq:zs}. We deduce that
    \begin{align*}
      &\E[(\hvlu-\vlu)^3\mid\fl]\\
      &\qquad = \E\left[\left(\frac1{n_1}\sum_{k
            \in \kl}Z_{\delta,k}\right)^3\mid\fl\right]\\
      &\qquad = \frac{O(1)}{n_1^3}\E\Bigg[\sum_{k \in \kl}Z_{\delta,k}^3  + \underset{k > k_1}{\sum_{k, k_1 \in \kl,}}Z_{\delta,k}^2Z_{\delta,k_1}\\
      &\qquad \qquad \qquad \qquad + \underset{k > k_1, k_2}{\sum_{k, k_1,k_2 \in \kl,}}Z_{\delta,k}Z_{\delta,k_1}Z_{\delta,k_2} \mid \fl\Bigg]\\
      &\qquad = O(n^{-\alpha_1}).
    \end{align*}

    For \(p = 4,\) by a similar argument, we have that for
    \(k,k_1,k_2,k_3\in K_l,\) \(k > k_1,k_2,k_3,\)
    \[\E[Z_{\delta,k}^4 \mid
    \fl],\,\E[Z_{\delta,k}^3Z_{\delta,k_1}\mid\fl],\,\E[Z_{\delta,k}^2Z_{\delta,k_1}^2\mid\fl] = O(1),\]
    \[\E[Z_{\delta,k}Z_{\delta,k_1}Z_{\delta,k_2}Z_{\delta,k_3}\mid\fl]
    = O(n^{-\alpha_1}),\]
    and if \(k_1 > k_2,\)
    \[\E[Z_{\delta,k}^2Z_{\delta,k_1}Z_{\delta,k_2}\mid\fl] = O(n^{-\alpha_1}).\]
    We thus obtain that
    \begin{align*}
      &\E[(\hvlu-\vlu)^4\mid\fl]\\
      &\qquad = \E\left[\left(\frac1{n_1}\sum_{k
            \in \kl}Z_{\delta,k}\right)^4\mid\fl\right]\\
      &\qquad = \frac{O(1)}{n_1^4}\E\Bigg[\sum_{k \in
        \kl}Z_{\delta,k}^4
      + \underset{k>k_1}{\sum_{k, k_1 \in \kl,}}Z_{\delta,k}^3Z_{\delta,k_1}\\
      &\qquad \qquad \qquad \qquad + \underset{k>k_1}{\sum_{k, k_1 \in
          \kl,}}Z_{\delta,k}^2Z_{\delta,k_1}^2
      + \underset{k>k_1>k_2}{\sum_{k, k_1,k_2 \in \kl,}}Z_{\delta,k}^2Z_{\delta,k_1}Z_{\delta,k_2}\\
      &\qquad \qquad \qquad \qquad + \underset{k>k_1,k_2,k_3}{\sum_{k, k_1,k_2,k_3 \in \kl,}}Z_{\delta,k}Z_{\delta,k_1}Z_{\delta,k_2}Z_{\delta,k_3} \mid \fl\Bigg]\\
      &\qquad = O(n^{-\alpha_1}).
    \end{align*}
  \item We first note that the quantity
    \begin{align*}
      \vlub &= \frac1{n_1}\sum_{k \in K_l} \E[\cos(u\hxk) \mid
      \fk]\\
      &=\frac1{n_1}\sum_{k\in K_l}\varphi_{k/n_0}(u) + O(n^{-1/4}),\\
      \intertext{using \autoref{lem:xcf},}
      &\ge 2\zeta(u) + O(n^{-1/4}),
    \end{align*}
    using \autoref{lem:tech2}(iv).  Then using Azuma's inequality, we have
    \begin{align*}
      &\P(\hvlu \le \zeta(u) \mid\fl)
      \\
      &\qquad \le \P(
      \hvlu - \vlub \le -\zeta(u) + O(n^{-1/4})\mid\fl)\\
      &\qquad \le \exp(-A'n^{1/8}),
    \end{align*}
    for a constant \(A' > 0.\) By a similar argument, we also have
    \[\P(\hplu \le \zeta(u) \mid\fl) \le \exp(-A'' n^{1/8}),\]
    for a constant \(A'' > 0.\) The result follows. \hfill\qedhere
  \end{enumerate}
\end{proof}

Finally, we may prove \autoref{thm:est}.

\begin{proof}[Proof of \autoref{thm:est}]
  From the definitions, we have that the estimates \(\hclu\) are
  \(\fls\)-measurable, and the events \(\el \in \fls.\) The bound on
  the probability of \(\elc\) likewise follows directly from
  \autoref{lem:bounds}(vi).

  It thus remains to prove the bounds on the mean and variance of
  \(\hclu.\) We will decompose the error in \(\hclu\) into three
  terms, controlling the error in each of \(\log(\hvlu),\)
  \(\log(\hplu),\) and \(\hulu.\)

  We first consider \(\log(\hvlu),\) and define the random variable
  \[Z_{\varphi,l} = \frac{\hvlu}{\vlu} - 1.\]
  We then have that
  \begin{align*}
    &(\log(\hvlu) - \log(\vlu))1(\el)\\ &\qquad = \log(1 + Z_{\varphi,l})1(\el)\\
    &\qquad = (Z_{\varphi,l} - \oh Z_{\varphi,l}^2 + \tfrac13 Z_{\varphi,l}^3 + O(Z_{\varphi,l}^4))1(\el),
  \end{align*}
  using Taylor's theorem, since on \(\el,\)
  \begin{equation}
    \label{eq:zlb}
    1 + Z_{\varphi,l} \ge \frac{\zeta(u)}{\vlu} \ge
    \zeta(u) > 0.
  \end{equation}

  To bound the error in \(\log(\hvlu),\) we will now take expectations
  of the \(Z_{\varphi,l}\) terms. We have that
  \begin{align*}
    \E[Z_{\varphi,l}1(\el) \mid \fl] 
    &= \E\left[\frac{\hvlu}{\vlu} - 1
    \mid \fl\right] + O(\P(\elc\mid\fl)),\\
    \intertext{since \(\hvlu\) is bounded, and \(\vlu \ge 2\zeta(u)
      > 0,\)}
    &= O(n^{-\alpha_1}),
  \end{align*}
  using \autoref{lem:bounds}(i) and (vi).  Similarly, we also have
  \begin{align*}
    \E[Z_{\varphi,l}^21(\el) \mid \fl] &= \ulu + O(n^{-\alpha_1}),\\
    \E[Z_{\varphi,l}^31(\el) \mid \fl] &= O(n^{-\alpha_1}),\\
    \E[Z_{\varphi,l}^41(\el) \mid \fl] &= O(n^{-\alpha_1}),
  \end{align*}
  using \autoref{lem:bounds}(i), (iii), (v) and (vi); as a consequence, we
  deduce
  \begin{align*}
    \E[\abs{Z_{\varphi,l}}^31(\el) \mid \fl] &\le \E[Z_{\varphi,l}^21(\el) \mid \fl]^{1/2}\E[Z_{\varphi,l}^41(\el) \mid
    \fl]^{1/2}\\
    &=O(n^{-\alpha_2}),
  \end{align*}
  using Cauchy-Schwarz.

  We can now bound the error in \(\log(\hvlu).\) We conclude that
\begin{align*}
&\E[(\log(\hvlu)-\log(\vlu))1(\el) \mid \fl]\\
&\qquad= \E[(Z_{\varphi,l} - \oh Z_{\varphi,l}^2 +
\tfrac13 Z_{\varphi,l}^3 + O(Z_{\varphi,l}^4))1(\el) \mid \fl]\\
&\qquad= -\oh\ulu + O(n^{-\alpha_1}),
\end{align*}
and similarly,
\begin{align*}
&\E[(\log(\hvlu)-\log(\vlu))^21(\el)\mid\fl]\\
&\qquad= \E[(Z_{\varphi,l}+O(Z_{\varphi,l}^2))^21(\el) \mid \fl]\\
&\qquad= \E[(Z_{\varphi,l}^2 + O(\abs{Z_{\varphi,l}}^3 + Z_{\varphi,l}^4))1(\el) \mid \fl]\\
&\qquad = \ulu + O(n^{-\alpha_2}).
\end{align*}

We next consider the error in \(\log(\hplu).\) By a similar argument,
we can obtain that
\begin{align*}
\E[(\log(\hplu)-\log(\plu))1(\el) \mid \fl]
&= O(n^{-\alpha_1}),\\
\E[(\log(\hplu)-\log(\plu))^21(\el) \mid \fl]
&= O(n^{-1/4}),
\end{align*}
using \autoref{lem:bounds}(ii), (iv) and (vi).

Finally, we prove bounds on \(\hulu,\) defining the random variable
\[Z_{\tau,l} = \widehat \varphi_l(2u) - \varphi_{l/n_2}(2u).\]
We then have
\begin{align*}
&(\hulu - \ulu)1(\el)\\
&\qquad = \frac1{n_1} \left( \frac{1 + \varphi_{l/n_2}(2u) +
    Z_{\tau,l}}{2\varphi_{l/n_2}^2(u)(1 + Z_{\varphi,l})^2} -\frac{1+\varphi_{l/n_2}(2u)}{2\varphi_{l/n_2}^2(u)}\right)1(\el)\\
&\qquad = 
\frac1{n_1}\left(\frac{-2(1 +
    \varphi_{l/n_2}(2u))Z_{\varphi,l} + Z_{\tau,l}}{2\varphi_{l/n_2}^2(u)}
+
  O(Z_{\varphi,l}^2 + Z_{\varphi,l}Z_{\tau,l})\right)1(\el),
\end{align*}
using \eqref{eq:zlb}, and that \(\vtu\) is bounded below.

Using \autoref{lem:bounds}(i), (iii) and (vi) as above, we also have that 
\begin{align*}
  \E[Z_{\tau,l}1(\el) \mid \fl] &= O(n^{-\alpha_1}),\\
  \E[Z_{\tau,l}^21(\el) \mid \fl] &= O(n^{-1/8}).
\end{align*}
We therefore conclude that
\begin{align*}
&\E[(\hulu-\ulu)1(\el)\mid\fl]\\
&\qquad = O(n^{-1/8})\E[Z_{\varphi,l} + Z_{\tau,l} + Z_{\varphi,l}^2
+ \abs{Z_{\varphi,l}}\abs{Z_{\tau,l}}\mid\fl],\\
\intertext{since \(\vtu\) is bounded below,}
&\qquad =O(n^{-1/4}),
\end{align*}
using Cauchy-Schwarz. Likewise,
\begin{align*}
\E[(\hulu - \ulu)^21(\el)\mid\fl]
&= O(n^{-1/4})\E[Z_{\varphi,l}^2 + Z_{\tau,l}^2\mid\fl]\\
&=O(n^{-3/8}).
\end{align*}

We have thus bounded the error in each of \(\log(\hvlu),\)
\(\log(\hplu),\) and \(\hulu.\) Combining these results, we deduce
that
\begin{align*}
&\E[(\hclu-\clu)1(\el) \mid \fl]\\
&\qquad = O(1)\E[((\log(\hvlu) - \log(\vlu)) + \oh\ulu\\
&\qquad \qquad
- (\log(\hplu) -
\log(\plu)) + \oh(\hulu - \ulu))1(\el) \mid \fl]\\
&\qquad = O(n^{-\alpha_1}),
\end{align*}
and
\begin{align*}
&\E[(\hclu-\clu)^21(\el) \mid \fl]\\
&\qquad = O(1)\E[((\log(\hvlu) - \log(\vlu))^2 + (\log(\hplu) - \log(\plu))^2\\
&\qquad \qquad + (\hulu - \ulu)^2 + O(n^{-1/4}))1(\el) \mid \fl]\\
&\qquad = \ulu + O(n^{-\alpha_2}).
\end{align*}
Finally, it can be checked that these results are uniform over \(l =
0, \dots, n_2-1,\) and \(\P \in \sacm.\)
\end{proof}

\subsection{Proofs of convergence rates}
\label{sec:proofs-con}

We next prove \autoref{thm:reg}, our result on the performance of our
regression estimate \(\tctu.\) Our argument follows from
\citet{tsybakov_introduction_2009}, taking care to account for the
extra error terms in the statement of \autoref{thm:est}, and the
stochastic nature of the target \(\ctu.\)

\begin{proof}[Proof of \autoref{thm:reg}]
  To begin, we will state some facts about local polynomial
  regression, as given in the proof of Theorem 1.7 in
  \citet{tsybakov_introduction_2009}. Since the design points
  \(l/n_2\) are uniform, we have that for large \(n,\) the matrices
  \(V_{n}(t)\) are invertible, and the weight functions \(W_{n,l}(t)\)
  well-defined. Furthermore, the weights \(W_{n,l}(t)\) satisfy:
  \begin{equation}
    \label{eq:wb} \abs{W_{n,l}(t)} =
    O\left(\frac1{n_2h}\right)1\left(\abs*{t-\frac{l}{n_2}}\le h\right),\end{equation}
  uniformly in \(l =0 , \dots, n_2-1;\)
  \begin{equation}
    \label{eq:ws} \sum_{l=0}^{n_2-1}\abs{W_{n,l}(t)} = O(1);
  \end{equation}
  and
  \begin{equation}
    \label{eq:wp} \sum_{l=0}^{n_2-1}\left(t-\frac{l}{n_2}\right)^p
    W_{n,l}(t) = \begin{cases} 1, &p=0,\\ 0, &p = 1, \dots, N-1.\end{cases}
  \end{equation}

  We now prove the results on our estimate \(\tctu.\) We must first
  define the high-probability event \(E\) given in the statement of
  the theorem. We let \(E_{a,b} = \bigcap_{l=a}^{b-1} E_l,\) and set
  \(E = E_{0,n_2}.\) We then note that from \autoref{thm:est}, we have
  \begin{align*}
    \P(E^c\mid\fz) &\le \sum_{l=0}^{n_2-1}\E[\P(\elc\mid\fl)\mid\fz]\\
    &= O(n^{3/8})\exp(-A n^{1/8})\\
    &\le \exp(-A' n^{1/8}),
  \end{align*}
  for constants \(A, A' > 0.\) Similarly, for \(l = 0, \dots,
  n_2-1,\) \(k \ge l,\) we have
  \begin{equation}
    \label{eq:pekc}
    \P(E_{k,n_2}^c\mid\fl) \le \exp(-A'n^{1/8}).
  \end{equation}

  We next split the estimates \(\hclu\) into bias and variance
  parts. Let
  \begin{equation}
    \label{eq:hcs}
    \hclu = \clu + \hcolu + \hctlu,
  \end{equation}
  where the bias term
  \[\hcolu = \frac{\E[(\hclu-\clu)1(E) \mid \fl]}{\P(E \mid \fl)},\]
  setting \(\hcolu = 0\) when \(\P(E \mid \fl) = 0,\) and the variance
  term \(\hctlu\) is then defined by \eqref{eq:hcs}.

  We can similarly split the regression estimates \(\tctu\) into bias
  and variance parts. Let
  \begin{equation}
    \label{eq:tctu}
    \tctu = \ctu + \tcotu + \tcttu + \tchtu,
  \end{equation}
  where the estimator bias and variance, \(\tcotu\) and \(\tcttu,\)
  are given by \[\tcktu = \sum_{l=0}^{n_2-1}W_{n,l}(t)\hcklu, \quad
  k = 1,2,\] and the regression bias
  \[\tchtu = \sum_{l=0}^{n_2-1}W_{n,l}(t)\clu - \ctu.\]

  To bound the error in \(\tctu,\) we must show that all three terms
  \(\tcktu\) are small. We begin with the estimator bias \(\tcotu,\)
  and note that for large \(n,\)
  \begin{align}
    \notag \abs{\hcolu} &= 1(E_{0,l})\frac{\E[(\hclu -
      \clu)1(E_{l,n_2})\mid\fl]}{\P(E_{l,n_2}\mid\fl)}\\
    \notag &= O(1)\E[(\hclu -
      \clu)1(E_{l,n_2})\mid\fl],\\
      \intertext{using \eqref{eq:pekc},}
\notag       &= O(1)(\E[(\hclu -
      \clu)1(\el)\mid\fl] + \P(E_{l+1,n_2}^c\mid\fl)),\\
\intertext{since \((\hclu -
      \clu)1(\el)\) is bounded,}
    \label{eq:hcolu} &=O(n^{-\alpha_1}),
  \end{align}
  using \eqref{eq:pekc} and \autoref{thm:est}. We thus have
  \[\abs{\tcotu} 
  \le \sum_{l=0}^{n_2-1}\abs{W_{n,l}(t)}\abs{\hcolu} =
  O(n^{-\alpha_1}),\] using \eqref{eq:ws} and \eqref{eq:hcolu}.

  We next consider the estimator variance \(\tcttu.\) We first note
  that
  \begin{align*}
    \E[\hctlu 1(E)\mid\fl] &= \E[(\hclu - \clu - \hcolu)1(E)\mid\fl]\\
    &= 0,
  \end{align*}
  and
  \begin{align*}
    \E[\hctlus 1(E)\mid\fl] &= O(1)(
\E[(\hclu - \clu)^21(E)\mid\fl] + \hcolus)\\
    &=
 O(n^{-1/8}),
  \end{align*}
  using \eqref{eq:hcolu} and \autoref{thm:est}. We thus have
  \begin{align*}
    \E[\tcttus 1(E)\mid \fz]
    &= \E\left[\left(\sum_{l=0}^{n_2-1}W_{n,l}(t)\hctlu 1(E)\right)^2\mid \fz\right]\\
    &= \sum_{l=0}^{n_2-1}W_{n,l}^2(t)\E[\hctlus 1(E)\mid\fz]\\
    &= O(n^{-1/8}) \left(\max_{l=0}^{n_2-1}\abs{W_{n,l}(t)}\right)\left(\sum_{l=0}^{n_2-1}\abs{W_{n,l}(t)}\right)\\
    &= O(n^{-2\alpha_3}),
  \end{align*}
  using \eqref{eq:wb} and \eqref{eq:ws}.

  Finally, we bound the regression bias \(\tchtu.\) Let \(m\) denote
  the largest integer smaller than \(\alpha.\) Using Taylor's
  theorem, for \(t \in [0,1],\) and \(l = 0, \dots, n_2-1,\) we then
  have that
  \[\clu = \ctu + \sum_{r=1}^{m-1}
  \frac{(t-l/n_2)^r}{r!}c_t^{(r)}(u) +
  \frac{(t-l/n_2)^m}{m!}c_{t_l}^{(m)}(u),\] for some \(t_l \in [0,1]\)
  lying between \(t\) and \(l/n_2.\) We deduce that
  \begin{align*}
    &\E[\tchtus\mid\fz]\\ &\qquad= \E\left[\left(\sum_{l=0}^{n_2-1}
        W_{n,l}(t)(\clu -\ctu)\right)^2\mid\fz\right],\\
    \intertext{using \eqref{eq:wp},}
    &\qquad = \E\left[\left(\sum_{l=0}^{n_2-1}
        W_{n,l}(t)\frac{(t-l/n_2)^m}{m!}(c_{t_l}^{(m)}(u) -
        c_t^{(m)}(u))\right)^2\mid\fz\right],\\
    \intertext{again using \eqref{eq:wp},}
    &\qquad =
    O(h^{2m})\sum_{k,l=0}^{n_2-1}\abs{W_{n,k}(t)}\abs{W_{n,l}(t)}1\left(\abs*{t
        - \frac{k}{n_2}}, \abs*{t - \frac{l}{n_2}} \le h\right)\\
    &\qquad \qquad \times \E[\abs{c_{t_k}^{(m)}(u) -
      c_t^{(m)}(u)}\abs{c_{t_l}^{(m)}(u) - c_t^{(m)}(u)}\mid\fz],\\
    \intertext{using \eqref{eq:wb},}
    &\qquad =
    O(h^{2\alpha})\left(\sum_{l=0}^{n_2-1}\abs{W_{n,l}(t)}\right)^2,\\
    \intertext{using Cauchy-Schwarz,}
    &\qquad= O(n^{-2\alpha_3}),
  \end{align*}
  using \eqref{eq:ws}.

  Combining these results, we obtain that
  \begin{align*}
    &\E[(\tctu - \ctu)^21(E)\mid\fz]\\
    &\qquad = O(1)\E[(\tcotus + \tcttus + \tchtus)1(E)\mid\fz]\\
    &\qquad = O(n^{-2\alpha_3}),
  \end{align*}
  as required. For the \(L^2\) risk, we likewise obtain
  \begin{align*}
    \E[\norm{\tctu-\ctu}_2^21(E)\mid\fz]
    &= \E\left[\int_0^1 (\tctu-\ctu)^21(E)\,dt\mid\fz\right]\\
    &= \int_0^1 \E[(\tctu-\ctu)^21(E)\mid\fz]\,dt\\
    &= O(n^{-2\alpha_3}).
  \end{align*}
  Finally, we can check that these rates are uniform over \(t \in
  [0,1],\) and \(\P \in \sacm.\)
\end{proof}

We may now deduce our corollary describing the performance of
\(\trtu\) and \(\tctu.\)

\begin{proof}[Proof of \autoref{cor:reg}]
  We first fix \(0 < C \le D,\) and prove bounds on the error of
  \(\tctu\) under the assumption that \(\P \in \sacm.\) For \(t \in
  [0,1],\) we have
  \[\abs{\tctu - \ct} \le \abs{\tctu - \ctu} + \abs{\ctu - \ct},\]
  and from \autoref{thm:reg},
  \[\abs{\tctu - \ctu} = O_p(n^{-\alpha_3}).\]
  
  It thus remains to bound \(\abs{\ctu - \ct}.\) We have that
  \begin{align*}
    \abs{\ctu - c_t} &= \frac{1}{n_0u^2}\int_\R\int_0^1 
    (1-\cos(\sqrt{n_0}\Phi(w)ux))\, dw\,\nu_t(dx)\\
    &= O(n_0^{-1})\int_\R (1 \wedge n_0x^2)\,\nu_t(dx)\\
    &= O(n_0^{-1})\int_\R (1 \wedge (n_0x^2)^{\beta/2}) \,\nu_t(dx)\\
    &= O(n_0^{-(2-\beta)/2})\int_\R (1 \wedge \abs{x}^\beta)\,\nu_t(dx)\\
    &= O(n_0^{-(2-\beta)/2}) = O(n^{-(2-\beta)/4)}).
  \end{align*}
  We thus conclude that
  \begin{equation}
    \label{eq:cer}
    \abs{\tctu - \ct} = O_p(n^{-\alpha_4});
  \end{equation}
  by a similar argument, the same holds also for the \(L^2\) error,
  \(\norm{\tcu-c}_2.\) We can further check that these limits hold
  conditionally on \(\fz,\) and uniformly over all \(t \in [0,1],\)
  \(\P \in \sacm.\)

  We next consider the case that \(\P \in \saecm,\) for some \(\gamma
  \in [0,1].\) Create, on an extended probability space, a process
  \(X^S_t,\) \(t \in [0,1],\) which almost-surely agrees with \(X_t\)
  at times \(t \in [0,S].\) For times \(t \in [S, 1],\) we require that
  \(X_t\) is a L\'evy process with respect to both \(\ft\) and
  \(\ftp,\) with characteristic triplet \((b_S,c_S,\nu_S).\)

  Also create observations
  \[Y^S_j = X^S_{j/n} + \ej^S,\quad j = 0, \dots, n-1,\] where for
  \(j/n \le S,\) the errors \(\ej^S = \ej.\) When \(j/n > S,\) we
  require that the errors \(\ej^S\) are \(\fjp\)-measurable, and equal
  to \(\pm\sigma_S\) each with probability \(\oh\) given
  \(\fj.\)

  Then let \(\tctsu\) denote the estimate of \(c_t\) defined similarly
  to \(\tctu,\) but using the observations \(Y^S_j.\) Conditionally on
  \(\te,\) the law of the \(X_t^S\) and \(Y_j^S\) lies in \(\sacm,\)
  so we can apply \eqref{eq:cer} to \(\tctsu.\) We obtain that
  \begin{equation}
    \label{eq:cer2}
    \abs{\tctsu - c_{t\wt}}1(\te) = O_p(f(C,D)n^{-\alpha_4}),
  \end{equation}
  uniformly in \(\gamma,\,C,\) and \(D,\) for a function \(f(C,D) >
  0.\) 
  
  We now consider the case \(\P \in \sa,\) and suppose we are given an
  arbitrary sequence \(\delta_n >0,\) \(\delta_n\to \infty.\) If we
  choose \(C_n \to 0,\) \(D_n \to \infty\) slowly enough as \(n \to
  \infty,\) we will obtain that \(f(C_n,D_n) = O(\delta_n).\)
  Since \(\P \in \sa,\) we also have that \(\P \in \mathcal
  S^{\alpha,\beta}_{\gamma_n}(C_n, D_n)\) for some \(\gamma_n \to 0;\)
  let \(\ten \in \fz\) and \(S_n \in [0,1]\) denote the associated
  events and stopping times. 

  Applying \eqref{eq:cer2}, we deduce that
  \[\abs{\widetilde c_t^{S_n}(u) - c_{t\wedge S_n}}1(\ten) = O_p(\delta_n n^{-\alpha_4}).\]
  Since \[\P(\ten \cap \{S_n = 1\}) \ge 1 - \gamma_n \to 1,\] this
  implies that
  \[\abs{\tctu - c_t} = O_p(\delta_n n^{-\alpha_4}).\]
  Since this result holds for any diverging sequence \(\delta_n,\) we
  conclude that
  \[\abs{\tctu - c_t} = O_p(n^{-\alpha_4}).\]
  Again, the result for the \(L^2\) error follows similarly.

  Next, we suppose that \(\P \in \tacm,\) and bound the accuracy of
  the estimate \(\trtu.\) We begin by bounding its normalising
  constant,
  \[\frac1{n_2}\sum_{l=0}^{n_2-1}\hclu = \sum_{l=0}^{n_2-1}\widetilde
  W_{n,l}(t)\hclu,\] where the weights \(\widetilde W_{n,l}(t) =
  1/n_2.\) Since these weights satisfy \eqref{eq:wb} and \eqref{eq:ws}
  for the bandwidth \(h = 1,\) we have that
  \begin{align*}
    \abs*{\frac1{n_2}\sum_{l=0}^{n_2-1}(\hclu -
    \clu)} &= \abs*{\sum_{l=0}^{n_2-1}\widetilde W_{n,l}(t)(\hclu -
    \clu)}\\
  &= O_p(n^{-\alpha_1}),
  \end{align*}
  arguing as in \autoref{thm:reg}.

  We also have
  \begin{align*}
    &\E\left[\left(\frac1{n_2}\sum_{l=0}^{n_2-1}\clu - \int_0^1 \ctu\,dt\right)^2\right]\\
    &\qquad= \E\left[\left(\sum_{l=0}^{n_2-1} \int_{l/n_2}^{(l+1)/n_2}
        (\clu
          -  \ctu)\,dt\right)^2\right]\\
    &\qquad\le \E\left[\sum_{l=0}^{n_2-1}
        \int_{l/n_2}^{(l+1)/n_2}(\clu - \ctu)^2\,dt\right],\\
      \intertext{by Jensen's inequality,}
    &\qquad= \sum_{l=0}^{n_2-1}\int_{l/n_2}^{(l+1)/n_2}\E[(\clu - \ctu)^2]
    \,dt\\
    &\qquad= O(n^{-2\alpha_1}).
  \end{align*}

  We thus deduce that
  \begin{equation}
    \label{eq:err-den}
    \abs*{\frac1{n_2}\sum_{l=0}^{n_2-1}\hclu - \int_0^1\ctu\,dt} =
  O_p(n^{-\alpha_1}).
  \end{equation}
  From \autoref{thm:reg}, we also have that
  \begin{equation}
    \label{eq:err-num}\abs{\tctu - \ctu} = O_p(n^{-\alpha_3}).
  \end{equation}
  Combining these results, we obtain that
  \[
    \abs{\trtu - r_t} = \abs*{\frac{\tctu}{\frac1{n_2}\sum_{l=0}^{n_2-1}\hclu} -
      \frac{\ctu}{\int_0^1 \ctu\,dt}} = O_p(n^{-\alpha_3}),\]
  since \(\int_0^1 \ctu\,dt \ge C > 0.\)

  We can again check that this limit holds conditionally on \(\fz,\)
  and uniformly over all \(t \in [0,1],\) \(\P \in \tacm.\)
  Arguing as above, we then conclude that for \(\P \in \ta,\)
  \[\abs{\trtu - r_t} = O_p(n^{-\alpha_3}).\]
  As above, we can also conclude that these results likewise hold for
  the \(L^2\) error \(\norm{\tru - r}_2.\)

  Finally, we bound the performance of \(\trtu\) for \(\P \in
  \sacm.\) Combining~\eqref{eq:cer},~\eqref{eq:err-den}
  and~\eqref{eq:err-num}, we have
  \[\abs*{\frac1{n_2}\sum_{l=0}^{n_2-1}\hclu - \int_0^1c_t\,dt},\,\abs{\tctu - c_t} =
  O_p(n^{-\alpha_4}).\]
  Arguing as above, we obtain that
  \[\abs{\trtu - r_t} = O_p(n^{-\alpha_4}),\]
  and that this can be extended to \(\P \in \sa,\) and the \(L^2\)
  error \(\norm{\tru-r}_2.\)
\end{proof}

Finally, we can also prove our lower bound on the rate of estimation,
which is a simple corollary of results in \citet{munk_lower_2010}.

\begin{proof}[Proof of \autoref{thm:lb}]
  We begin with part (i), and appeal to the proof of Theorem 2.1 in
  \citet{munk_lower_2010}. The authors give a lower bound on the
  \(L^2\) estimation rate of \(c_t,\) in a setting similar to our
  \(\sacm.\)

  \citeauthor{munk_lower_2010} consider a setting where \(\sigma^2_t =
  \sigma^2 > 0\) is a deterministic constant, \(c_t\) is
  deterministic, and \(b_t = \nu_t = 0.\) They then construct a large
  number of choices \(c_{\omega,t}\) for the volatility, separated
  from each other in \(L^2\) norm at a rate at least
  \(n^{-\alpha_3}.\) They further establish that, given observations
  \(Y_j\) under one such volatility function \(c_{\omega,t},\) we
  cannot consistently estimate \(\omega.\) They thus show that no
  estimate \(c_t^*\) of \(c_t\) can satisfy \(\norm{c^* - c}_2 =
  o_p(n^{-\alpha_3}).\)

  It can be checked that, when \(C < 1 < D,\) their models lie in
  \(\sacm \cap \mathcal T\) for large \(n,\) so their lower bound
  holds also in that setting. By rescaling their volatility functions
  \(c_{\omega,t},\) we can obtain the same results also for general
  \(0 < C < D.\)

  A pointwise lower bound can be proved by a similar argument; we
  sketch a proof below. Define two choices for the volatility,
  \[c_{0,t} = 1, \qquad
  c_{1,t} = 1 + h^{\alpha} K((1-t)/h),\] where \(h =
  n^{-1/2(2\alpha+1)},\) and \(K : \R \to \R\) is a smooth
  non-increasing non-negative function, satisfying \(K(0) = 1,\)
  \(K(1) = 0.\)

  We note that when \(C < 1 < D,\) these models lie within \(\sacm\)
  for large \(n;\) as above, by rescaling we can work with general \(0
  < C < D.\) We also have that \(c_{0,1}\) and \(c_{1,1}\) are
  separated at a rate \(n^{-\alpha_3}.\) It thus suffices to show that
  we cannot consistently distinguish \(c_0\) from \(c_1\) given the
  \(Y_j.\)

  We begin by moving to a more informative model, where we
  additionally observe one efficient price \(X_t,\) at a time \(t =
  \lfloor (1-h)n \rfloor / n.\) Given \(X_t,\) the observations
  \(Y_j,\,j \le nt\) are independent of the \(Y_j,\,j > nt;\)
  furthermore, the former are identically distributed under \(c_0\)
  and \(c_1.\)

  We therefore need consider only the observations \(X_t\) and
  \(Y_j,\,j > nt.\) Arguing similarly to \citeauthor{munk_lower_2010},
  it can be shown that these observations are insufficient to
  distinguish \(c_0\) and \(c_1,\) thereby establishing our lower
  bound.

  For part (ii), it can be checked that the rate functions
  \(r_{\omega,t} = c_{\omega,t}/\int_0^1c_{\omega,s}\,ds\) are again
  separated, in \(L^2\) norm or pointwise, at a rate at least
  \(n^{-\alpha_3}.\) We thus conclude that our lower bounds hold also
  for \(r_t.\)
\end{proof}


\subsection{Technical proofs}
\label{sec:pro}

We now give proofs of our technical lemmas.  We begin with a simple
proof of \autoref{lem:ihm}, our result establishing that c\`agl\`ad
It\=o semimartingales with locally bounded characteristics satisfy our
assumptions.

\begin{proof}[Proof of \autoref{lem:ihm}]
  For \(D > 0,\) we define events \(\te\) and stopping times \(S\)
  with the desired properties.  When \(D < 2,\) we may set \(\te =
  \emptyset.\) When \(D \ge 2,\) we set \(\te = \{\abs{Z_0} \le D\},\)
  and
  \begin{multline*}
    S = \sup\Bigg\{s \in [0,1]: \abs{Z_s} \le R,\\\abs{b_{Z,s}} + \int_{1 \le \abs{x} <
        3R}\abs{x}\,\nu_{Z,s}(dx),\,
    c_{Z,s} + \int_{\abs{x} < 3R} x^2\,\nu_{Z,s}(dx) \le D/2
    \Bigg\},
  \end{multline*}
  where \(R \in [0,D]\) is to be determined.  We must then show that
  on the event \(\te,\) \(Z_t \in \ihm,\) with \(\P(\te \cap \{S =
  1\}) \to 1\) as \(D \to \infty.\)

  On \(\te,\) we first note that since \(Z_t\) is c\`agl\`ad, the
  condition \(\abs{Z_{t\wt}} \le D\) follows from the definitions of
  \(R\) and \(S.\) To establish \(Z_t \in \ihm,\) we then split
  \(Z_t\) into a predictable term,
  \[Z_{P,t} = \int_0^{t^-} b_{Z,s}\,ds + \int_0^{t^-}\int_{1 \le \abs{x}
    < 3R} x\,\nu_{Z,s}(dx)\,ds,\]
  martingale term
  \[Z_{M,t} = \int_0^{t^-} \sqrt{c_{Z,s}}\,dB_{Z,s} +
  \int_0^{t^-}\int_{\abs{x} < 3R} x\,(\mu_Z(dx,ds) -
  \nu_{Z,s}(dx)\,ds),\]
  and large-jump term
  \[Z_{J,t} = \int_0^{t^-}\int_{\abs{x} \ge 3R}x\,\mu_Z(dx,ds).\]

  We thus have
  \[Z_t = Z_{P,t} + Z_{M,t} + Z_{J,t}.\] Furthermore, the stopped
  process \[Z_{t\wt} = Z_{P,t\wt} + Z_{M,t\wt},\] since the stopped
  large-jump term \(Z_{J,t\wt} = 0.\) We thus deduce that, for \(0 \le
  t \le s \le 1,\)
  \begin{align*}
    &\E[(Z_{s\wt} - Z_{t\wt})^2\mid\ftp]\\
    &\qquad \le 2\E[(Z_{P,s\wt} - Z_{P,t\wt})^2 + 
      (Z_{M,s\wt} - Z_{M,t\wt})^2 \mid \ftp]\\
      &\qquad \le 2\E\Bigg[\left(\int_{s\wt}^{t\wt}\left(\abs{b_{Z,u}}
            + \int_{1 \le \abs{x} < 3R}
            \abs{x}\,\nu_{Z,u}(dx)\right)\,du\right)^2\\
        &\qquad \qquad + \int_{s\wt}^{t\wt}\left(c_{Z,u} +
          \int_{\abs{x} < 3R} z^2\,\nu_{Z,u}(dx)\right)\,du \mid \ftp\Bigg]\\
      &\qquad \le D^2(s-t)^2/2 + D(s-t)\\
      &\qquad \le D^2(s-t),
  \end{align*}
  using the definition of \(S,\) and that \(D \ge 2.\)

  We conclude that \(Z_t \in \ihm,\) so it remains to show that
  \(\P(\te \cap \{S = 1\}) \to 1\) as \(D \to \infty.\) We first
  consider the event \(\te,\) and note that since \(Z_0\) is finite, we have
  \(\P(\te) \to 1\) as \(D \to \infty.\)

  We next consider the stopping time \(S.\) As the integrals
  \(\int_{\abs{x} < R} x^2\,\nu_{Z,t}(dx)\) are locally bounded, we
  have that if \(R \to \infty\) slowly enough with \(D,\) then
  \[\P\left(\sup_{t \in [0,1]}\int_{\abs{x} < R}x^2 \,\nu_{Z,t}(dx)
    > \frac{D}4\right) \to 0\]
  as \(D \to \infty.\) Since \(b_t\) and \(c_t\) are also locally
  bounded, we likewise have
  \[\P\left(\sup_{t \in [0,1]}\abs{b_t} >
    \frac{D}4\right),\,\P\left(\sup_{t \in [0,1]}c_t >
    \frac{D}4\right) \to 0.\]
  Finally, as \(X_t\) is c\`agl\`ad, it is again locally
  bounded, and
  \[\P\left(\sup_{t \in [0,1]}\abs{X_t} > R\right) \to 0\]
  as \(D \to \infty.\)
  Combining these results, we obtain that
  \(\P(S = 1) \to 1\)
  as \(D \to \infty,\) as required.
\end{proof}

We next establish our technical lemmas \hyperref[lem:tech2]{Lemmas}
\ref{lem:tech2}--\ref{lem:hssk}, used in the proof of
\autoref{thm:est}; we begin with some new definitions. We will write
the characteristic functions of the log-prices \(X_t\) in terms of the
spot characteristic exponent,
\[\phantomsection \label{def:theta} \theta_t(u) = ib_t u - \oh c_t u^2 +
\int_\R \left(e^{iux} - 1 - iux1_{\abs{x} < 1}\right) \nu_t(dx).\]

We will also describe the pre-averaged increments \(\hxk\) using
constants \(\aj, \bj.\) For \(j \in J_k,\) \(k = 0,\dots, n_0-1,\) we
define
\[\phantomsection \label{def:pj} \aj = \begin{cases} \Phi_n(j/n), &j+1 \in J_k,\\ 0,
  &\text{otherwise},\end{cases}\]
and set \(p_{-1} = 0.\) (Note that this does not conflict with
our earlier definition of \(\aj,\) which held only for \(j,j+1\in
J_k.\)) Then for \(j =0, \dots, n-1,\) we define
\[\bj = p_{j-1} - \aj.\]
We may now proceed with the lemmas.

\begin{lemma}
  \label{lem:tech}
  In the setting of \autoref{thm:est}, let \(0 \le t \le s \le 1,\)
  \(u, v \in \R,\) and \(k = 0, \dots, n_0-1.\) We then have:
  \begin{enumerate}
  \item \(\abs{\theta_t(u)} = O(1 + u^2);\)
  \item \(\abs{\theta_t(u) - \theta_t(v)} =
    O(1+(\abs{u}+\abs{v})\abs{u-v});\)
  \item \(\E[\abs{\theta_s(u) - \theta_t(u)}^2\mid\ftp] = O(1 +
    u^2 + u^4(s-t)^{2\alpha_0});\)
  \item \(\sum\tjik \bj^2 = 2\kappa + O(n^{-1/2});\) and
  \item \(\int_{k/n_0}^{(k+1)/n_0} \theta_{k/n_0}(\Phi_n(w)u)\,dw =
    -c_{k/n_0}(u)u^2.\)
  \end{enumerate}
\end{lemma}

\begin{proof}
  We prove each statement in turn.
  \begin{enumerate}
  \item For \(t \in [0,1],\) \(u, x \in \R,\) we have
    \begin{align}
      \notag\abs{e^{iux} - 1 - iux1_{\abs{x} < 1}} &\le
      \oh\abs{ux}^21_{\abs{x} <
        1} + 2\cdot1_{\abs{x} \ge 1}\\
      \label{eq:f-small} &\le 2(1+u^2)(1 \wedge x^2),
    \end{align}
    so
    \begin{align*}
      \abs{\theta_t(u)} &= O(1 + u^2)\left(\abs{b_t} + \abs{c_t} +
        \int_\R (1 \wedge x^2)\, \nu_t(dx)\right)\text{\hfill}\\
      &= O(1 + u^2).\end{align*}
  \item For \(t \in [0,1],\) \(w \in \R,\) \(\abs{x} < 1,\) we
    likewise have
    \[\abs{ix(e^{iwx}-1)} \le \abs{wx^2},\]
    so for \(u,v \in \R,\)
    \begin{align*}
      \abs{\theta_t(u)-\theta_t(v)}
      & \le \int_u^v \abs*{\frac{d}{dw}\left(\theta_t(w) -
          \int_{\abs{x} \ge 1} (e^{iwx} -1)\,\nu_t(dx)\right)}\,dw \\
      &\qquad+ \int_{\abs{x} \ge 1} \abs{(e^{iux}-1)-(e^{ivx}-1)}\,\nu_t(dx)\\
      &= \int_u^v \abs*{ib_t - wc_t +
        \int_{\abs{x} < 1}  ix(e^{iwx} - 1)\,\nu_t(dx)}\,dw \\
      &\qquad+ \int_{\abs{x} \ge 1} \abs{e^{iux} - e^{ivx}}\,\nu_t(dx)\\
      &=  \int_u^vO(1+\abs{w})\left(\abs{b_t} + \abs{c_t} + \int_{\abs{x} < 1}
        x^2 \,\nu_t(dx)\right)\,dw\\
      &\qquad + O(1)\int_{\abs{x} \ge 1} \nu_t(dx)\\
      &= O(1 + (\abs{u}+\abs{v})\abs{u-v}).
    \end{align*}
  \item For \(u, x \in \R,\) we first set
    \[f(x) = \frac{e^{iux}-1-iux1_{\abs{x}<1}}{2(1+u^2)},\] 
    and for \(t \in [0,1],\)
    \[Z_{\nu,t} = \int_\R f(x)\,\nu_t(dx).\]
    From
    \eqref{eq:f-small}, we have that \(\abs{f(x)} \le (1 \wedge
    x^2),\) so if \(\beta > 1,\) \(Z_{\nu,t} \in \iam,\) and for
    \(0 \le t \le s \le 1,\)
    \[\E[(Z_{\nu,s}-Z_{\nu,t})^2\mid\ftp]  = O((s-t)^{2\alpha_0}).\]

    If \(\beta \le 1,\) we likewise have
    \[
      \abs{f(x)} \le \frac{\abs{ux}1_{\abs{x} < 1} + 1_{\abs{x} \ge
          1}}{1+u^2}
      \le \frac{2(1\wedge \abs{x})}{1 + \abs{u}},\]
    so \(\abs{Z_{\nu,t}} = O(1/(1+\abs{u})),\)
    and
    \[(Z_{\nu,s}-Z_{\nu,t})^2 \le 2(Z_{\nu,s}^2 +
    Z_{\nu,t}^2) = O(1/(1+u^2)).\]

    In either case, since
    \[\theta_t(u) = ib_tu - \oh c_tu^2 + 2(1+u^2)Z_{\nu,t},\]
    we deduce that
    \begin{align*}
      &\E[\abs{\theta_s(u) - \theta_t(u)}^2\mid\ftp]\\
      &\qquad = O(u^2)\E[(b_s-b_t)^2\mid\ftp] + O(u^4)\E[(c_s-c_t)^2\mid\ftp]\\
      &\qquad \qquad +
      O(1+u^4)\E[(Z_{\nu,s} - Z_{\nu,t})^2 \mid \ftp]\\
      &\qquad = O(1 + u^2 + u^4(s-t)^{2\alpha_0}).
    \end{align*}
  \item For \(k = 0,\dots,n_0-1,\) we have
    \begin{align*}
      \sum\tjik \bj^2 &= \sum\tjik\left([\Phi_n(t)]_{(j-1)/n}^{j/n}\right)^2 + O(n^{-1/2}),\\
      \intertext{since for \(j \in J_k,\) \(\bj =
        -[\Phi_n(t)]_{(j-1)/n}^{j/n}\) unless \(j-1\not\in J_k\) or
        \(j+1\not\in J_k,\) in which case both \(\bj\) and
        \([\Phi_n(t)]_{(j-1)/n}^{j/n}\) are \(O(n^{-1/4}),\)} &=
      \sum\tjik\left(\int_{(j-1)/n}^{j/n}
        \Phi_n'(t)\,dt\right)^2 + O(n^{-1/2})\\
      &= n^{-2}\sum\tjik\Phi_n'(j/n)^2 + O(n^{-1/2}),\\
      \intertext{since for \(\abs{s-t} \le 1/n,\) \(\abs{\Phi_n'(s) -
          \Phi_n'(t)} = O(n^{1/4}),\)} &=
      n^{-1}\int_{k/n_0}^{(k+1)/n_0}\Phi_n'(t)^2\,dt + O(n^{-1/2})\\
      &= 2\kappa + O(n^{-1/2}).
    \end{align*}
  \item For \(k = 0, \dots, n_0-1,\) we have
    \begin{align*}
      \int_{k/n_0}^{(k+1)/n_0} \theta_{k/n_0}(\Phi_n(w)u)\,dw
      &= \frac1{n_0}\int_0^1
      \theta_{k/n_0}(\sqrt{n_0}\Phi(w)u)\,dw\\
      &= \Re\left( \frac1{n_0}\int_0^1
      \theta_{k/n_0}(\sqrt{n_0}\Phi(w)u)\,dw\right),\\
      \intertext{since \(\theta_t\) is Hermitian, and \(\Phi\) is
        antisymmetric about \(\oh,\)}
      &= -c_{k/n_0}(u)u^2. \qedhere
    \end{align*}
\end{enumerate}
\end{proof}


We may now prove \autoref{lem:tech2}.

\begin{proof}[Proof of \autoref{lem:tech2}]
  We consider each process \(\xi_t\) in turn, proving (iv) as a
  corollary to (i).
  \begin{enumerate}
  \item For \(u, x \in \R,\) we set
    \[f(x) = \frac{1}{2n_0u^2}\int_0^1 (1 -
    \cos(\sqrt{n_0}\Phi(w)ux))\,dw,\]
    and note that
    \[0 \le f(x) \le \frac{1 \wedge n_0 u^2x^2}{n_0u^2}
    \le 1 \wedge x^2.\]

    If \(\beta > 1,\) the process
    \begin{align*}
      Z_{c,t} &= \int_\R f(x) \,\nu_t(dx)
    \end{align*}
    is thus in \(\iam,\) so for \(0 \le t \le s \le 1,\)
    \[\E[(Z_{c,s}-Z_{c,t})^2\mid\ftp] = O((s-t)^{2\alpha_0}).\]

    If instead \(\beta \le 1,\) we likewise have
    \[0 \le f(x) \le \frac{1 \wedge \sqrt{n_0}\abs{ux}}{n_0u^2}
    \le \frac{1 \wedge \abs{x}}{\sqrt{n_0}\abs{u}},\]
    so \(\abs{Z_{c,s}} = O(n^{-1/4}),\) and
    \[(Z_{c,s}-Z_{c,t})^2 \le 2(Z_{c,s}^2 + Z_{c,t}^2) = O(n^{-1/2}).\]

    In either case, since
    \[c_t(u) = c_t + 2Z_{c,t}(u),\]
    we then have that
    \begin{align*}
      \E[(c_s(u) - c_t(u))^2 \mid \ftp] &= O(1)\E[(c_s-c_t)^2 + (Z_{c,s}-Z_{c,t})^2\mid\ftp]\\
      &=O((s-t)^{2\alpha_0} + n^{-1/2}).
    \end{align*}
    Since
    \[Z_{c,t} \le \int_\R 1 \wedge \abs{x}^\beta\,\nu_t(dx)
    \le D,\] we also have \(c_t(u) \le 3D,\) almost surely.
  \item For \(0 \le t \le s \le 1,\) we have
    \begin{align*}
      &\E[(\varphi_s(u) - \varphi_t(u))^2 \mid \ftp]\\
      &\qquad \le 2\E[(c_s(u) - c_t(u))^2 + (\sigma^2_s(u) - \sigma^2_t(u))^2
      \mid \ftp]\\
      &\qquad = O((s-t)^{2\alpha_0} + n^{-1/2}),
    \end{align*}
    using (i).
  \item The result follows similarly to (ii). \hfill \qedhere
  \end{enumerate}
\end{proof}

We next describe the characteristic functions of the increments of
\(X_t\) and noises \(\ej.\) For \(j = 0, \dots, n-2,\) define the
increments
\[\phantomsection \label{def:dx} \dxj = X_{(j+1)/n} - X_{j/n},\]
and set \(\Delta X_{n-1} = 0.\) We then have the following result.

\begin{lemma}
  \label{lem:cfs}
  In the setting of \autoref{thm:est}, let \(j = 0, \dots, n-1,\) \(u
  \in \R.\) Then
  \begin{enumerate}
  \item if \(u = o(1),\)
    \[\E[\exp(iu\ej)\mid\fj] = \exp(-\oh\sigma^2_{j/n}u^2) +
    O(\abs{u}^3); \quad \text{or}\]
  \item
    if \(j \ne n-1,\) and \(u = O(n^{1/2}),\)
    \begin{align*}
      \E[\exp(iu\dxj)\mid\fjp]
      &= \exp(n^{-1} \theta_{j/n}(u)) +
      O(n^{-1}(1+\abs{u}+u^2n^{-\alpha_0})).
    \end{align*}
  \end{enumerate}
\end{lemma}

\begin{proof}
  We prove each result in turn.
  \begin{enumerate}
  \item We note that \(\ej \mid \fj\) has bounded fourth
  moment, so we can expand its characteristic function to third order
  using Taylor's theorem. We obtain that
  \begin{align*}
    &\E[\exp(iu\ej)\mid\fj]\\
    &\qquad = 1 + iu\E[\ej\mid\fj] -
    \hus\E[\ej^2\mid\fj] + O(\abs{u}^3)\E[\abs{\ej}^3 \mid \fj]\\
    &\qquad = 1 - \oh\sigma^2_{j/n}u^2 + O(\abs{u}^3)\\
    &\qquad = \exp(-\oh\sigma^2_{j/n}u^2) + O(\abs{u}^3),
  \end{align*}
  for small enough \(u,\) since \(\sigma^2_{j/n}\) is bounded.
\item We define
  \[\Theta_t = \int_{0}^t \theta_s(u)\,ds,\quad t \in [0,1],\] and note
  from \autoref{lem:tech}(i) that \(\theta_s(u)\) is bounded.  The
  process \(\Theta_t\) is thus bounded, continuous and of finite
  variation. We deduce that its stochastic exponential,
  \[\mathcal E(\Theta)_t = \exp(\Theta_t) \ne 0.\]
  From Corollary II.2.48 in \citet{jacod_limit_2003}, we then have
  that the process
  \[M_{X,t} = \frac{\exp(iuX_t)}{\mathcal E(\Theta)_t}\] is a local
  \(\ftp\)-martingale; since \(M_{X,t}\) is bounded, it is also a true
  martingale.

  We can thus use \(M_{X,t}\) to compute the characteristic functions
  of the increments \(\dxj.\) From \autoref{lem:tech}(i), we have that
  for \(u = O(n^{1/2}),\)
  \[
    \frac{M_{X,(j+1)/n}}{M_{X,j/n}} = \exp\left(iu\dxj -
      \int_{j/n}^{(j+1)/n} \theta_s(u)\,ds\right)
    = O(1),\]
  and similarly, the random variable
  \[Z_{X,j} = \int_{j/n}^{(j+1)/n}(\theta_s(u) -
  \theta_{j/n}(u))\,ds = O(1).\]

  We therefore obtain that
  \begin{align*}
    &\exp(-n^{-1}\theta_{j/n}(u))\E[\exp(iu\dxj) \mid
    \fjp]\\
    &\qquad =
    \E\left[\frac{M_{X,(j+1)/n}}{M_{X,j/n}}\exp(Z_{X,j}) \mid
    \fjp\right]\\
    &\qquad = 1 +
    \E\left[\frac{M_{X,(j+1)/n}}{M_{X,j/n}}(\exp(Z_{X,j})-1) \mid
      \fjp\right],\\
    \intertext{since \(M_{X,t}\) is a martingale,}
    &\qquad = 1 + O(1)\E[\abs{Z_{X,j}}\mid\fjp],\\
    \intertext{since \(M_{X,(j+1)/n}/M_{X,j/n}\) and \(Z_{X,j}\) are bounded,}
    &\qquad = 1 +
    O(1)\E\left[\int_{j/n}^{(j+1)/n}\abs{\theta_s(u)-\theta_{j/n}(u)}\,ds\mid\fjp\right]\\
    &\qquad = 1 + O(1)\int_{j/n}^{(j+1)/n}\E[\abs{\theta_s(u)-\theta_{j/n}(u)}^2\mid\fjp]^{1/2}\,ds\\
    &\qquad = 1 + O(n^{-1}(1 + \abs{u} + u^2n^{-\alpha_0})),
  \end{align*}
  using \autoref{lem:tech}(iii). The result follows since
  \(n^{-1}\theta_{j/n}(u)\) is bounded, using \autoref{lem:tech}(i).
  \hfill\qedhere
\end{enumerate}
\end{proof}

We may now prove \autoref{lem:xcf}.



\begin{proof}[Proof of \autoref{lem:xcf}]
  We first note we may assume \(n\) is large enough that \(J_k\) is
  non-empty. We then express the distribution of the pre-averaged
  increments \(\hxk\) in terms of the increments \(\dxj,\) and noises
  \(\ej.\) From the definitions, we obtain
  \begin{align*}
    \hxk &= \sum_{j\in J_k} \aj (\dxj -\ej + \ejp)\\
    &= \sum_{j\in J_k} (\aj \dxj + \bj \ej).
  \end{align*}

  We now compute the characteristic functions of this sum, conditional
  on \(\fk;\) we begin by writing down the characteristic functions of
  the terms \(\aj\dxj.\)
  For \(j \in J_k,\) set
  \[Z_{\theta,j} = \abs{\theta_{j/n}(\aj u)-\theta_{k/n_0}(\aj u)}.\]
  We then have that for \(j, j + 1 \in J_k,\)
  \begin{align}
    \notag &\E[\exp(iu\aj\dxj)\mid \fjp] \\
    \notag &\qquad = \exp(n^{-1}
    \theta_{j/n}(\aj u))+O(n^{-3/4}),\\
    \intertext{since \(\abs{\aj} = O(n^{1/4}),\) using \autoref{lem:cfs}(ii),}
    \label{eq:adxcf} &\qquad = \exp(n^{-1}\theta_{k/n_0}(\aj u)) +
    O(n^{-3/4} + n^{-1} Z_{\theta,j}),
  \end{align}
  since by \autoref{lem:tech}(i), \(n^{-1}\theta_{t}(\aj u)\) is
  bounded. The result \eqref{eq:adxcf} also holds when \(j \in J_k,\)
  \(j+1 \not \in J_k,\) since then \(\aj = 0.\)

  We can also write down the characteristic functions of the terms
  \(\bj\ej.\) For \(j \in J_k,\) set
  \[Z_{\sigma,j} = \abs{\sigma^2_{j/n}-\sigma^2_{k/n_0}}.\] We then
  have that
  \begin{align}
    \notag &\E[\exp(iu\bj\ej)\mid \fj]\\
    \notag &\qquad = \exp(-\oh \bj^2 \sigma^2_{j/n} u^2) + O(n^{-3/4}),\\
    \intertext{since \(\abs{\bj} = O(n^{-1/4}),\) using
      \autoref{lem:cfs}(i),}
    \label{eq:becf} &\qquad = \exp(-\oh \bj^2 \sigma^2_{k/n_0} u^2) +
    O(n^{-3/4} + n^{-1/2} Z_{\sigma,j}),
  \end{align}
  since \(\oh \sigma^2_{t} u^2\) is bounded.
  
  We may thus compute inductively the characteristic functions of the
  sums
  \[\hxkm = \sum_{j \in J_k\,:\, j \ge m}
  (\aj\dxj + \bj\ej).\]
  By induction on \(m,\) we will show that
  \begin{align}
    \notag &\E[\exp(iu\hxkm)\mid \fm]\\
    \label{eq:cfd} &\qquad = \exp\left(\sum_{j \in J_k\,:\,j \ge m}(n^{-1}\theta_{k/n_0}(\aj u) - \oh\bj^2\sigma^2_{k/n_0}u^2)\right)\\
    \notag &\qquad \qquad +O(1)\sum_{j \in J_k\,:\,j \ge m}
    \E[n^{-3/4} + n^{-1}Z_{\theta,j} + n^{-1/2}Z_{\sigma,j}\mid\fm].
  \end{align}
  
  In the base case, when \(m = \max(J_k) + 1,\) the
  result is trivial. In the inductive case, we assume that \(m \in
  J_k,\) and \eqref{eq:cfd} holds for \(m+1.\) Since
  \begin{align*}
    \E[\exp(iu\hxkm) \mid \fm]
    &= \E[\exp(iu\bm\emm)\E[\exp(iu\am\dxm)\\
    &\qquad \times
    \E[\exp(iu\widehat X_{k,m+1})\mid\fms]\mid \fmp]\mid\fm],
  \end{align*}
  and using \eqref{eq:adxcf} and \eqref{eq:becf}, we have that
  \eqref{eq:cfd} holds also for \(m.\)

  We therefore have that \eqref{eq:cfd} holds when \(m = \min(J_k),\) in
  which case \(\hxkm = \hxk.\) We conclude
  that
  \begin{align*}
    &\E[\cos(u\hxk)\mid\fk]\\
    &\qquad = \Re\left(\E[\exp(iu\hxk)\mid\fk]\right)\\
    &\qquad = \Re\left(\E[\E[\exp(iu\hxkm)\mid\fm]\mid\fk]\right)\\
    &\qquad = \Re\left(\exp\left(\sum_{j \in
        J_k}(n^{-1}\theta_{k/n_0}(\aj u) -
      \oh\bj^2\sigma^2_{k/n_0}u^2)\right)\right)+ O(n^{-1/4}),\\
    \intertext{using \eqref{eq:cfd} and \autoref{lem:tech}(iii),}
    &\qquad = \Re\left(\exp\left(\int_{k/n_0}^{(k+1)/n_0}
      \theta_{k/n_0}(\Phi_n(w)
      u)\,dw - \kappa \sigma^2_{k/n_0}u^2\right)\right) + O(n^{-1/4}),\\
    \intertext{since for \(\abs{s-t} \le 1/n,\) \(\abs{\Phi_n(s) -
        \Phi_n(t)} = O(n^{-1/4}),\) and using \autoref{lem:tech}(i), (ii)
      and (iv),} &\qquad= \vku + O(n^{-1/4}),
  \end{align*}
  using \autoref{lem:tech}(v). As a consequence, we also obtain
  \begin{align*}
    &\Var[\cos(u\hxk) \mid \fk]\\
    &\qquad = \E[\cos(u\hxk)^2 \mid \fk] - \E[\cos(u\hxk) \mid \fk]^2\\
    &\qquad = \oh \E[1 + \cos(2u\hxk) \mid \fk] - \E[\cos(u\hxk)
    \mid
    \fk]^2\\
    &\qquad = \oh(1 + \varphi_{k/n_0}(2u) + O(n^{-1/4})) - (\vku +
    O(n^{-1/4}))^2,\\
    &\qquad = \rho^2_{k/n_0}(u) + O(n^{-1/4}),
  \end{align*}
  since \(\vtu\) is bounded.
\end{proof}

We next move on to our bounds on the noise estimates \(\hssk.\) We
will first need to decompose the log-price process \(X_t\) into two
parts: we set
\[\phantomsection \label{def:xij} X_t = X_{I,t} + X_{J,t},\] where the square-integrable component
\[X_{I,t} = \int_{0}^t b_s \,ds + \int_{0}^t
\sqrt{c_s}\, dB_s + \int_{0}^t \int_{\abs{x} < 1} x\,
(\mu(dx,ds)-\nu_s(dx)\,ds),\]
and the large-jump component
\[X_{J,t} = \int_{0}^t \int_{\abs{x} \ge 1} x\,\mu(dx,ds).\] We begin
by proving a technical result on the variation of
the process \(X_{I,t}.\)

\begin{lemma}
  \label{lem:xm}
  In the setting of \autoref{thm:est}, for \(j = 0, \dots, n-1,\) \(p
  = 2, 4,\) we have
  \[\E\left[(X_{I,(j+1)/n}-X_{I,j/n})^p \mid \fjp\right] = O(n^{-1}).\]
\end{lemma}

\begin{proof}
  Our argument follows \citet{luschgy_moment_2008}. We define the
  \(\ftp-\)martingale
  \[M_{I,t} = X_{I,t} - \int_0^tb_s\,ds,\] and note that
  \begin{align*}
    &\E\left[(X_{I,(j+1)/n}-X_{I,j/n})^p\mid\fjp\right]\\
    &\qquad = O(1)
    \E\left[(M_{I,(j+1)/n}-M_{I,j/n})^p +
      \left(\int_{j/n}^{(j+1)/n}\abs{b_s}\,ds\right)^p\mid\fjp\right]\\
    &\qquad = O(1)\E\left[(M_{I,(j+1)/n}-M_{I,j/n})^p\mid\fjp\right] + O(n^{-p}),
  \end{align*}
  so it suffices to prove an equivalent bound for \(M_{I,t}.\)

  If \(p = 2,\) we note that
  \begin{align}
    \notag &\E[(M_{I,(j+1)/n}-M_{I,j/n})^2\mid\fjp]\\
    \notag &\qquad = \E[[M_I]_{(j+1)/n}-[M_I]_{j/n}\mid\fjp]\\
    \notag &\qquad = \E\left[\int_{j/n}^{(j+1)/n} c_s\,ds + \int_{j/n}^{(j+1)/n}\int_{\abs{x} < 1} x^2 \,\mu(dx,
      ds)\mid\fjp\right]\\
    \notag &\qquad =\int_{j/n}^{(j+1)/n}\E\left[c_s + \int_{\abs{x} < 1} x^2
      \,\nu_s(dx)\mid\fjp\right] ds\\
    \label{eq:miv} &\qquad =O(n^{-1}),
  \end{align}
  as required.

  If instead \(p = 4,\) then since the quadratic variation \([M_I]_t\)
  is integrable, we may define the martingale
  \begin{equation}
    \label{eq:mqd}
    M_{V,t} = [M_I]_t - \E[[M_I]_t].
  \end{equation}
  We then note that
  \begin{align}
    \notag &\E[(M_{V,(j+1)/n} - M_{V,j/n})^2 \mid \fjp]\\
    \notag &\qquad = \E[[M_V]_{(j+1)/n}-[M_V]_{j/n}\mid\fjp]\\
    \notag &\qquad = \E\left[\int_{j/n}^{(j+1)/n}\int_{\abs{x} < 1}
      x^4\,\mu(dx,ds)\mid\fjp\right],\\
    \intertext{since \([M_V]_t\) depends only on the jumps in \(X_{I,t},\)}
    \notag &\qquad = \int_{j/n}^{(j+1)/n}\E\left[\int_{\abs{x} < 1}
      x^4\,\nu_s(dx)\mid\fjp\right] ds\\
    \notag &\qquad \le \int_{j/n}^{(j+1)/n}\E\left[\int_{\abs{x} < 1}
      x^2\,\nu_s(dx)\mid\fjp\right] ds\\
    \label{eq:mqv} &\qquad = O(n^{-1}).
  \end{align}

  We thus have that
  \begin{align*}
    &\E[(M_{I,(j+1)/n}-M_{I,j/n})^4\mid\fjp]\\
    &\qquad = O(1)\E[([M_I]_{(j+1)/n}-[M_I]_{j/n})^2\mid\fjp],\\
    \intertext{using the Burkholder-Davis-Grundy inequality,}
    &\qquad = O(1)\E[(M_{V,(j+1)/n} - M_{V,j/n} +
    O(n^{-1}))^2\mid\fjp]\\
    \intertext{using \eqref{eq:miv} and \eqref{eq:mqd},}
    &\qquad = O(1)\E[(M_{V,(j+1)/n}-M_{V,j/n})^2\mid\fjp] + O(n^{-2})\\
    &\qquad = O(n^{-1}),
  \end{align*}
  using \eqref{eq:mqv}.
\end{proof}

We may now prove \autoref{lem:hssk}.



\begin{proof}[Proof of \autoref{lem:hssk}]
  We first denote by \(E_{\sigma,k}\) the event that \(X_{J,t},\) the
  large-jump component of the log-price, is constant over the interval
  \([k,k+1)/n_0.\) Since the expected number of jumps in \(X_{J,t}\)
  over that interval is
  \[\int_{k/n_0}^{(k+1)/n_0} \int_{\abs{x} \ge 1} \,\nu_s(dx)\,ds
  = O(n^{-1/2}),\]
  we have that \(E_{\sigma,k}\) holds with high probability,
  \[\P(E_{\sigma,k}^c) = O(n^{-1/2}).\]

  On the event \(E_{\sigma,k},\) we have that \(X_{J,t}\) makes no
  contribution to our estimate \(\hssk.\) In this case, \(\hssk\) will
  be equal to
  \[\tssk = \frac{n_0}{2n}\sum_{j,j+1\in J_k}(\nj + \xij)^2,\]
  where the random variables
  \[\nj = \ejs - \ej, \qquad \xij = X_{I,(j+1)/n} - X_{I,j/n}.\]
  
  Since \(E_{\sigma,k}\) holds with high probability, we may therefore
  proceed by bounding \(\tssk.\) We set
  \[S_k = \tssk - \ssk,\] and then have
  \[S_k = S_{k,0} + S_{k,1} + S_{k,2} + S_{k,3} + O(n^{-1/2}),\]
  for the sums
  \begin{align*}
    S_{k,0} &= \frac{n_0}{2n}\sum\tjmik (\sigma^2_{j/n} + \sigma^2_{(j+1)/n} - 2\sigma^2_{k/n_0}),\\
    S_{k,1} &= \frac{n_0}{2n}\sum\tjmik (\nj^2 - \sigma^2_{j/n} - \sigma^2_{(j+1)/n}),\\
    S_{k,2} &= \frac{n_0}{n}\sum\tjmik \nj\xij,\\
    S_{k,3} &= \frac{n_0}{2n}\sum\tjmik \xij^2,
  \end{align*}
  which we will bound in turn.

  To bound \(S_{k,0},\) we note that if \(j,j+1\in J_k,\) we have
  \begin{align*}
    &\E[(\sigma^2_{j/n} + \sigma^2_{(j+1)/n} - 2\sigma^2_{k/n_0})^2 \mid
    \fk]\\
    &\qquad = O(1)\E[(\sigma^2_{j/n}-\sigma^2_{k/n_0})^2 + (\sigma^2_{(j+1)/n} - \sigma^2_{k/n_0})^2 \mid
    \fk]\\
    &\qquad = O(n^{-1/2}),
  \end{align*}
  so \(\E[S_{k,0}^2 \mid \fk] = O(n^{-1/2}).\) Similarly, to bound
  \(S_{k,1},\) we note that if also \(j_1, j_1+1 \in J_k,\) we have
  \begin{align*}
    &\E[(\nj^2 - \sigma^2_{j/n} - \sigma^2_{(j+1)/n})(Z_{\varepsilon,j_1}^2 -
    \sigma^2_{j_1/n} - \sigma^2_{(j_1+1)/n}) \mid \fk]\\
    &\qquad = \begin{cases}
      O(1), &\abs{j - j_1} \le 1,\\
      0, &\text{otherwise}.\end{cases}
  \end{align*}
  so
  \(\E[S_{k,1}^2 \mid \fk] = O(n^{-1/2}).\)

To bound \(S_{k,2},\) we note that
\begin{align*}
  \E[\nj^2\xij^2 \mid \fk] &= \E[\ej^2\E[\xij^2 \mid \fjp] + \xij^2\E[\ejs^2 \mid \fjs]
 \mid \fk]\\
&= O(n^{-1}),
\end{align*}
using \autoref{lem:xm}, so \(\E[S_{k,2}^2 \mid \fk] = O(n^{-1}).\)
For \(S_{k,3},\) we likewise have
\[\E[\xij^4 \mid \fk] = O(n^{-1}),\]
using \autoref{lem:xm}, so \(\E[S_{k,3}^2\mid\fk] = O(n^{-1}).\)
From the above, we may deduce that
\begin{align*}
  \E[S_k^2 \mid \fk] &= O(1)\E[S_{k,0}^2 + S_{k,1}^2 + S_{k,2}^2 +
S_{k,3}^2 \mid \fk] + O(n^{-1})\\
&= O(n^{-1/2}),
\end{align*}
and thus \(\E[\abs{S_k} \mid \fk] = O(n^{-1/4}).\)

We have thus controlled the deviation of \(S_k;\) it remains to
compute its moment generating function,
\[f(v) = \E[\exp(-vS_k)\mid\fk].\] Since \(S_k \ge -\ssk,\) for \(v
\ge 0\) we may take derivatives under the expectation, obtaining that
\begin{align*}
  \abs{f'(v)} &= \abs{\E[S_k\exp(-vS_k)\mid\fk]}\\
  &\le
\exp(v\ssk)\E[\abs{S_k}\mid\fk]\\
&= O(n^{-1/4}),
\end{align*}
uniformly over \(v \in [0,u],\) for fixed \(u \ge 0.\)

From Taylor's theorem, we then have that for some \(v \in [0, u],\)
\[f(u) = f(0) + uf'(v) =1 + O(n^{-1/4}).\]
We thus deduce that
\begin{align*}
&\E[\exp(-u(\hssk-\ssk))\mid\fk]\\
&\qquad= \E[\exp(-u(\hssk-\ssk))1(E_{\sigma,k})\mid\fk] +
O(n^{-1/2}),\\
\intertext{since \(\hssk \ge 0,\) and \(\sigma^2_t\) is bounded,}
&\qquad= \E[\exp(-uS_k)1(E_{\sigma,k})\mid\fk] + O(n^{-1/2})\\
&\qquad=f(u) + O(n^{-1/2}),\\
\intertext{since \(S_k\) is bounded below,}
&\qquad=1 + O(n^{-1/4}).
\end{align*}

We have thus computed the moment generating function of
\(\hssk-\ssk;\) we may now deduce our results.  Since \(\ptu\) is
bounded, we conclude that
\begin{align*}
\E[\exp(-\kappa \hssk u^2)\mid\fk] &= \pku + O(n^{-1/4}),
\end{align*}
and
\begin{align*}
  &\Var[\exp(-\kappa \hssk u^2)\mid\fk]\\
  &\qquad = \E[\exp(-2\kappa \hssk u^2)\mid\fk] -
  \E[\exp(-\kappa \hssk u^2)\mid\fk]^2\\
  &\qquad = \pksu(1 + O(n^{-1/4}) - 1 +
  O(n^{-1/4}))\\ 
  &\qquad = O(n^{-1/4}),
\end{align*}
as required.
\end{proof}

\section*{Notations}
\label{sec:not}

\Needspace{1cm}

\newcommand{\e}[3]{\(#1\) & #2 \unskip \dotfill \pageref{#3}\\}
\newcommand{\h}[1]{\multicolumn{2}{c}{\bf #1}\\\\}
\begin{center}
\begin{longtabu} to \textwidth {cX}
\mainheader
\e{X_t}{efficient price process}{eq:sm}
\e{b_t,c_t,\nu_t}{semimartingale characteristics of
  \(X_t\)}{eq:sm}
\e{B_t}{Brownian motion in \(X_t\)}{eq:sm}
\e{\mu(dx, dt)}{Poisson random measure in \(X_t\)}{eq:sm}
\e{\ej}{microstructure noises}{eq:ms}
\e{Y_j}{observed prices}{eq:ms}
\e{n}{number of observations}{eq:ms}
\e{L_t}{L\'evy process in time-changed model}{eq:tc}
\e{b,c,\nu}{L\'evy characteristics of \(L_t\)}{eq:tc}
\e{R_t}{time-change process}{eq:tc}
\e{r_t}{normalised volatility}{eq:def-r}
\e{n_0,n_1,n_2}{numbers of bins used in estimates}{def:n}
\e{h_1,h_2}{bandwidths in \(n_1, n_2\)}{def:n}
\e{\Phi,\Phi_n}{scaling functions for pre-averaging}{def:phi}
\e{\hxk}{pre-averaged increments}{def:hx}
\e{J_k}{index sets for pre-averaging}{def:hx}
\e{\hssk}{estimated noise variance}{def:hs}
\e{\hvlu}{estimated characteristic functions}{def:hv}
\e{K_l}{index sets for volatility estimation}{def:hv}
\e{\hplu}{estimated noise characteristic function}{def:hp}
\e{\kappa}{constant in noise characteristic function}{def:hp}
\e{\hclu}{estimated volatilities}{def:hc}
\e{\hulu}{bias-correction term in \(\hclu\)}{def:ht}
\e{\ctu}{adjusted volatility process}{def:cu}
\e{\tctu}{local-polynomial estimate of \(c_t\)}{def:lp}
\e{K,N,h}{parameters in \(\tctu\)}{def:lp}
\e{W_{n,l}(t)}{weights in \(\tctu\)}{def:lpw}
\e{\trtu}{local-polynomial estimate of \(r_t\)}{def:hr}
\e{T}{time horizon}{def:time}
\e{\ft,\ftp}{filtrations of \(X_t\) before and after noise}{def:fp}
\e{\mathcal S}{base class of probability measures \(\P\)}{def:ss}
\e{\iam}{class of \(\alpha\)-smooth processes}{def:iam}
\e{\alpha_0}{Lipschitz smoothness rate}{def:iam}
\e{\saecm}{class of semimartingale models}{def:scm}
\e{\Omega_0, S}{probable event and stopping time in \(\saecm\)}{def:scm}
\e{\sacm,\sa}{semimartingale models given by \(\saecm\)}{def:sc}
\e{\tacm,\ta}{classes of time-changed L\'evy models}{def:ta}
\e{\vtu,\ptu}{asymptotic means of \(\hvlu,\hplu\)}{def:vt}
\e{\ttu,\utu}{asymptotic variances of \(\hvlu,\hclu\)}{def:tt}
\e{\alpha_1,\alpha_2}{convergence rates in mean and variance of \(\hclu\)}{def:ao}
\e{\alpha_3}{convergence rate for \(\tctu\) estimating \(\ctu\)}{def:at}
\e{\alpha_4}{convergence rate for \(\tctu\) estimating \(c_t\)}{def:af}
\e{E_l, \zeta(u)}{probable events and constant in \autoref{thm:est}}{def:el}
\suppheader
\e{\theta_t}{spot characteristic function of \(X_t\)}{def:theta}
\e{p_j,q_j}{weights for pre-averaging, and their
  differences}{def:pj}
\e{\dxj}{increments of \(X_t\)}{def:dx}
\e{X_{I,t},X_{J,t}}{integrable and large-jump components of
  \(X_t\)}{def:xij}
\end{longtabu}
\end{center}


\bibliographystyle{abbrvnat}
{\footnotesize \bibliography{tcnlm}}

\end{document}